\documentclass[leqno, a4, fleqn, 12pt]{amsart}

\usepackage{hyperref}
\usepackage{comment}
\usepackage{dsfont}
\usepackage{amsmath}
\usepackage{amssymb}
\usepackage{graphicx}

\usepackage{amssymb}

\usepackage[all,cmtip]{xy}

\usepackage{amsfonts}

\usepackage{soul}

\usepackage{mathpazo}
\usepackage{color}
\usepackage{yfonts}

\usepackage{paralist}

\usepackage{stmaryrd}

\usepackage{amsxtra}

\def\cA{\mathcal A}

\newcommand{\norm}[1]{\| #1\|}

\def\a{          \alpha}

\def\cA{          \mathcal A}

\def\bbA{          \mathbb A}

\def\cB{          \mathcal B}

\def\cW{          \mathcal W}

\def\tcW{          \widetilde{\mathcal W}}

\let\cal\mathcal

\def \R{{\mathbb R}}

\def \Z{{\mathbb Z}}
\def \N{{\mathbb N}}

\def \a{{\alpha}}
\def \ta{{\tilde\alpha}}

\def \tb{{\tilde\beta}}

\newcommand{\prf}{{\begin{proof}}}
\newcommand{\epf}{{\end{proof}}}

\newcommand{\bg}{{\mathbf g}}

\newcommand{\cE}{{\mathcal E}}

\newcommand{\tcE}{{\widetilde{\mathcal E}}}

\newcommand{\C}{{\mathbb C}}

\newcommand{\liealgg}{{\mathfrak g}}
\newcommand{\liealga}{{\mathfrak a}}

\newcommand{\liealgm}{{\mathfrak m}}

\newcommand{\ary}{\begin{eqnarray}}
\newcommand{\eary}{\end{eqnarray}}

\newcommand{\aryst}{\begin{eqnarray*}}
\newcommand{\earyst}{\end{eqnarray*}}

\newcommand{\enmt}{\begin{enumerate}}
\newcommand{\eenmt}{\end{enumerate}}

\DeclareMathOperator{\diff}{Diff}

\newtheorem{theo}{\sc Theorem}
\newtheorem{prop}{\sc Proposition}[section]

\newtheorem{lemma}{\sc lemma}[section]

\newtheorem{claim}{\sc claim}

\newtheorem{cor}{\sc corollary}

\newtheorem{conj}[theo]{\sc Conjecture}

\theoremstyle{definition}

\def\bee{\begin{equation}}
\def\eee{\end{equation}}

\newtheorem{defi}{\sc Definition}[section]

\theoremstyle{rema}

\newtheorem{rema}{\sc Remark}

\newcommand{\pdvr}[2]
{\dfrac{\partial^{#2} #1}{\partial \theta^{#2_1} \partial r^{#2_2}}}
\newcommand{\pdvrs}[2]
{\partial^{#2} #1 /\partial \theta^{#2_1} \partial r^{#2_2}}
\newtheorem{thm}{\sc Theorem}%[section]

\numberwithin{equation}{section}

\author{Zhiyuan Zhang}

\address{Institute for Advanced Study}
\email{zzzhangzhiyuan@gmail.com}

\begin{document}

\title[Zimmer's conjecture for lattice actions: the ${\rm SL}(n, \C)$-case]{Zimmer's conjecture for lattice actions: the ${\rm SL}(n, \C)$-case}

\date{\today}

\maketitle

\begin{abstract}
We prove Zimmer's conjecture for co-compact lattices in ${\rm SL}(n, \C)$: for any co-compact lattice in ${\rm SL}(n, \C)$, $n \geq 3$,  any $\Gamma$-action on a compact manifold $M$ with dimension: (I) less than $2n-2$ if $n \neq 4$, (II) less than $5$ if $n = 4$,  by $C^{1+\epsilon}$ diffeomorphisms factors through a finite action.
\end{abstract}

\section{Introduction}

%\begin{thm}\label{thm: main}
%Let $n \geq 3$ be an integer, and let $\Gamma < SL(n, \C) $ be a co-compact lattice. Let $M$ be a connected, compact manifold with $\dim M < 2n-2$. Then any group homomorphism $\alpha: \Gamma \to \diff^{1+\epsilon}(M)$ has finite image.
%\end{thm}

Motivated by a sequence of results  on the rigidity of linear representations including \cite{Sel, Weil, Mostow, Pra},  Margulis' superrigidity theorem 
 \cite{Mar}, and the extension to cocycles, Zimmer's cocycle superrigidity theorem \cite{Zim}, R. Zimmer proposed the following conjecture.
\begin{conj} \label{conj finite}
Let $G$ be a connected, semisimple Lie group with finite center, all of whose almost-simple factors have reak-rank at least $2$. Let $\Gamma < G$ be a lattice. Let $M$ be a compact manifold. If $\dim M < \min(n(G), d(G), v(G))$ then any homomorphism $\a : \Gamma \to \diff(M)$ has finite image.
\end{conj}
In the above conjecture, number $n(G)$ denotes the minimal dimension of a non-trivial real representation of the Lie algebra $\liealgg$ of $G$; number $v(G)$ denotes the minimal codimension of a maximal (proper) parabolic subgroup of $Q$ of $G$; and number $d(G)$ denotes the minimal dimension of all non-trivial homogeneous space $K/ C$ as $K$ varies over all compact real-forms  of all simple factors of the complexification of $G$.
There are also Zimmer's conjectures for volume-preserving actions. We refer the readers to \cite[Conjecture 1.2]{BFH1} for the statement of the full Zimmer's conjecture as extended by Farb and Shalen. We refer the readers to \cite{Fis1, Fis2} for the history of Zimmer's program as well as recent developments. 

In a recent breakthrough \cite{BFH1}, Brown, Fisher and Hurtado have  proved the non-volume preserving case of Zimmer's conjecture for co-compact lattices in higher-rank split simple Lie groups as well as certain volume preserving cases (under $C^2$ regularity assumption).
In \cite{BFH2}, the authors proved Zimmer's conjecture for the non-uniform lattice $SL(n, \Z)$. In \cite{DZ}, the authors replaced the regularity assumption $C^2$ in \cite{BFH1} by $C^1$ under a stronger dimensional constrain. We also mention \cite{Ye} for $SL(n, \Z)$ actions by homeomorphisms under a topological condition on the manifold.

 For many non-split Lie groups, the results in \cite{BFH1} also give  dimensional bounds that are comparable to the optimal bounds. For instance, for $n \geq 5$, the dimensional bound in \cite{BFH1} for $SL(n, \C)$, $SL(n, \mathbb H)$ are respectively one half and one quarter of the optimal bounds.
 In this paper, we improve the bound for $SL(n, \C)$ to the optimal level for co-compact lattices.
 The following are the main results of this paper.
\begin{thm}\label{thm: main2}
Let $n \geq 3$ be an integer, and let $\Gamma < SL(n, \C) $ be a co-compact lattice. Let $M$ be a connected, compact manifold satisfying: (I) $\dim M < 2n-2$ if $n \neq 4$, (II) $\dim M < 5$ if $n = 4$. Then any group homomorphism $\alpha: \Gamma \to \diff^{1+\epsilon}(M)$ factors through a finite group.
\end{thm}

\begin{thm}\label{thm: main3}
	Let $n \geq 3$ be an integer, and let $\Gamma < SL(n, \C) $ be a co-compact lattice. Let $M$ be a connected, compact manifold satisfying: 1. $\dim M < 2n-2$. Then any group homomorphism $\alpha: \Gamma \to \diff^{2}(M)$ preserves a Riemannian metric.
\end{thm}

\subsection{Further extensions}
The method of this paper can be generalized to other simple complex Lie group as well. In an on-going joint work with Jinpeng An, we will address Conjecture \ref{conj finite} for all simple complex Lie groups. This will appear as a second version of this paper.

\begin{comment}
Theorem \ref{thm: main2} is optimal, as the projective action of any lattice in $SL(n, \C)$ on $\mathbb{P}(\C^{n})$ is faithful, and thus improves the main result in \cite{BFH1} for  $SL(n, \C)$.
\end{comment}

\subsection*{Notation}
For any positive integer $m$, we denote by $[m]$ the set $\{1 \cdots, m\}$.
For any metric space $Z$, we use $\cB_Z$ to denote the Borel $\sigma$-algebra of $Z$, and use $\cal M (Z)$ to denote the set of Radon measures on $Z$. Given a measurable partition $\xi$, we denote by $\cB_{\xi}$ the $\sigma$-algebra generated by $\xi$.

\section{Preliminary}

Let $M$ be a connected, compact manifold.

Let $G = SL(n, \C)$ and let $\liealgg = sl(n, \C)$.

Let $H$ be the standard Cartan subgroup of $G$, i.e., $H$ is the subgroup of diagonal matrices in $G$.
We have $H = MA$ where 
$A$ is the subgroup  consisted of positive real diagonal matrices in  $G$; and $M$ is the subgroup consisted of diagonal matrices in $G$ with unit complex numbers 
on the diagonal.

  For each $1 \leq i, j \leq n$, let $E_{i,j}$ denote the $n \times n$-matrix whose entry at $i$-th row $j$-th column equals $1$, and $0$ at all other places.
We can see that the Lie algebra of $A$ and $M$ are respectively,
\aryst
\liealga =\{\sum_{i=1}^n a_i E_{i,i} \mid \sum_{i=1}^n a_i = 0, a_i \in \R \} \quad \mbox{and} \quad \liealgm = i\liealga.
\earyst
For any linear functional $\ell$ on $\liealga$, we denoted by $[\ell]$ the set of linear functionals on $\liealga$ which are positively proportional to $\ell$.
We let $\Sigma$ be the set of coarse restricted roots of $G$.
In our case, the coarse restricted roots are in bijection with the restricted roots. We will however adopt this notion in \cite{BFH1} to facilitate the citation of certain theorems.
We can show that $\Sigma = \{ [\gamma_{i,j}]  \mid 1 \leq i \neq j \leq n \}$ where we set $\gamma_{i,j}  = E_{i,i}^{*} - E_{j,j}^*$.
When there is no confusion, we slightly abuse the notion and write $\chi$ instead of $[\chi]$, for instance, we say that
 the root space for $\gamma_{i,j}$ equals $\C E_{i,j}$, which we denote by $\liealgg^{\chi_{i,j}}$. 
For each $\chi \in \Sigma$, we denote by $G^{\chi}$ the root subgroup of $\chi$, and denote by $\nu_{G^\chi}$ the Haar measure on $G^\chi$.
Also we denote $L_\chi = {\rm Ker}(\chi)$, and let $H_{\chi}$ denote the subgroup of $A$ corresponding to $L_\chi$.
We denote  $\Sigma^+ = \{ \gamma_{i,j} \mid 1 \leq i < j  \leq n \}$ and  $\Sigma^- = \{ \gamma_{i,j} \mid 1 \leq j < i \leq n \}$.
 We let $P$ denote the Borel subgroup of $G$ relative to our choice of $\Sigma^+$, i.e., the subgroup consisted of upper triangluar matrices. 
It is clear that $P$ is generated by $A, M$ and $G^{\chi}$, $\chi \in \Sigma^+$.

\subsection{Suspension space}

Let $\Gamma$ be a co-compact lattice in $G$. 
Let $\alpha: \Gamma \to \diff^{1+\epsilon}(M)$ be an right action, i.e., $\alpha(gh) = \a(h)\a(g)$.
 As in \cite{BFH1}, we consider the right $\Gamma$-action
 \aryst
 (g,x) \cdot \gamma = (g\gamma, \a(\gamma)(x))
 \earyst
 and the left $G$-action 
 \aryst
 a \cdot (g,x) = (ag,x).
 \earyst
 Let  $ M^\a = (G \times M) / \Gamma$, and let $\ta$ denote the left $G$-action on $M^\a$.
To simply notation, we will abbreviate $\ta(\exp(k))$ as $\ta(k)$ for every $k \in \liealga$.
We denote the canonical projection from $M^\a$ to $G /\Gamma$ by $\pi$.

Let $\mu$ be an $A$-invariant $A$-ergodic measure on $M^\a$.  
For any $k \in \liealga$, for $\mu$-a.e. $x$, we denote by  $\cW^{-}_{\ta(k)}(x)$, resp. $\cW^{+}_{\ta(k)}(x)$, the stable manifold, resp. unstable manifold, through $x$ for the map $\ta(k)$. 

For each $\chi \in \Sigma$, we define $E^{\chi}$, $E^{\chi}_F$, $E^{\chi}_G$, $\cW^{\chi}$, $\cW^{\chi}_F$ and $\cW^{\chi}_G$ as in \cite{BRHW2}.  For example, we have
\aryst
\cW^{\chi}(x) = \bigcap_{k \in \liealga, \chi(k) < 0} \cW^{-}_{\ta(k)}(x).
\earyst
It is clear that  $\dim E^{\chi}_F$ is $\mu$-a.e. constant.

\subsection{Conditional measure}\label{subsec: Conditional measure}

In this section, we define a collection of equivalence classes of measures $\{[\mu^{\cW^\chi}_x] \}_{x \in M^\a}$ where each $\mu^{\cW^\chi}_x$ is a measure defined up to a scalar with the property that $\mu^{\cW^\chi}_x$ is supported on $\cW^\chi(x)$.
Moreover, this collection is invariant under the $A$-action.

Let $\xi$ be a measurable partition subordinate to $\cW^\chi$. We let $\{ \mu^\xi_x \}_{x \in M^\a}$ denote the conditional measure associate to $\xi$. Let $\xi_1, \xi_2$ be two measurable partitions subordinate to $\cW^\chi$. Then for $\mu$-a.e. $x$, the restrictions of $\mu^{\xi_1}_x$ and $\mu^{\xi_2}_x$ to $\xi_1(x) \cap \xi_2(x)$ coincides up to a factor.

Take $k_0 \in \liealga$ such that $\chi(k_0) > 0$, and take $f=\ta(k_0)$. 
We take $\xi$, an $f$-increasing measurable partition subordinate to $\cW^\chi$.
We take an arbitrary precompact open neighborhood of $x$ in $\cW^\chi$, denoted by $U$. For $\mu$-a.e. $x$, we define 
\aryst
\mu^{\cW^\chi}_x = \lim_{n \to \infty} [\mu^{f^{n}(\xi)}_x(U)]^{-1} \mu^{f^{n}(\xi)}_x.
\earyst
It is direct to verify that the definition of $\mu^{\cW^\chi}_x$ is independent of the choice of $\xi$.
We say that two Radon measures $\zeta_1$, $\zeta_2$ on $\cW^\chi$ are equivalent if there is $c > 0$ such that $\zeta_2 = c\zeta_1$. Given a Radon measure $\zeta$ on $\cW^\chi$, we denote by $[\zeta]$ the equivalence class of $\zeta$. We notice that $[\mu^{\cW^\chi}_x]$ is independent of the choice of $U$.

By the $A$-invariance of $\mu$, we claim that for any $k \in \liealga$, for $\mu$-a.e. $x$, we have
\ary \label{item: invarianceconditionalmeasure}
[D\ta(k)_*\mu^{\cW^\chi}_x]  = [ \mu^{\cW^\chi}_{\ta(k)(x)} ].
\eary

We define $\{ [\mu^{\cW^\chi_F}_x] \}_{x \in M^\a}$ in an analogous way. We can see that $\{ [\mu^{\cW^\chi_F}_x] \}_{x \in M^\a}$ is also $A$-invariant.

\subsection{Coarse restricted root}
Given an $A$-invariant, $A$-ergodic measure $\mu$,
we consider the following subsets of $\Sigma$:
\aryst
\Sigma^{out} &=& \{\chi \in \Sigma \mid  E^{\chi}_F \neq \{0\} \}, \\
\Sigma^{out}_1 &=& \{ \chi \in \Sigma^{out} \mid \dim E^{\chi}_F \geq 2\}, \\
\Sigma^{out}_2 &=& \{ \chi \in \Sigma^{out} \mid \dim E^{\chi}_F =1, \dim E^{-\chi}_F \geq 1  \}, \\
\Sigma^{out}_3 &=& \Sigma^{out} \setminus (\Sigma^{out}_1 \cup \Sigma^{out}_2).
\earyst
We notice that the above subsets can also be defined for any $H$-ergodic measure $\mu$.
Indeed, we can define the above subsets of $\Sigma$ for each $A$-ergodic component of $\mu$. As $M$ is compact and commutes with $A$, $\dim E^\chi_*$ and $\Sigma^{out}_*$ are the same for all $A$-ergodic components of $\mu$ (see the paragraph below \cite[Theorem 5.8]{BFH1}).

It is clear that 
\ary \label{lab: sigmaout12}
2|\Sigma^{out}_1 \cup \Sigma^{out}_2| + |\Sigma^{out}_3| \leq \dim M.
\eary

Given a closed subgroup $Q \subset G$ containing $H$.
 We define
\ary \label{defsigmaq}
\Sigma_Q = \{ \chi \in \Sigma \mid  G^{\chi} \subset Q \}.
\eary
By \cite[Proposition 5.1]{BRHW2}, we have
\ary \label{lab: sigmasigmaotusigmaq}
\Sigma \setminus \Sigma^{out} \subset \Sigma_{Q}
\eary
for $Q = \{g \in G \mid g_* \mu = \mu \}$.

The following proposition\footnote{Lemma \ref{prop. para} is proved by Jinpeng An.} plays an important role in our proof.
\begin{lemma}\label{prop. para}
Let $Q$ be a closed subgroup of $G$ such that $H \subset Q$. If $n \geq 2$ and we have
\aryst
	|\Sigma \setminus \Sigma_Q| < 2n-2, 
\earyst
then $Q$ is a parabolic subgroup of $G$.
\end{lemma}
\begin{proof}
In view of \cite[Page 92, Prop. 11]{Bo}, it suffices to verify $\Sigma_Q \cup (-\Sigma_Q) = \Sigma$, i.e., for every $\gamma_{i,j} \in \Sigma$, either $\gamma_{i,j}$ or $\gamma_{j,i}$ lies in $\Sigma_Q$. To show this, consider the following $2n-2$ mutually disjoint sets:
\aryst
\{\gamma_{i,j}\}, \{\gamma_{j,i}\}, \{\gamma_{i,k}, \gamma_{k,i}\}, \{ \gamma_{j,k}, \gamma_{k,i} \}, \quad k \in [n] \setminus \{i,j\}. 
\earyst 
It follows from the assumption that at least one of such sets is contained in $\Sigma_Q$. If $\{\gamma_{i,j}\}$ or $\{\gamma_{j,i}\}$ is contained in $\Sigma_Q$, then there is nothing to prove. If $\{\gamma_{i,k}, \gamma_{k,j}\} \subset \Sigma_Q$, then $\gamma_{i,j} = \gamma_{i,k} + \gamma_{k,j} \in \Sigma_Q$. Similarly, if $\{\gamma_{j,k}, \gamma_{k,i}\} \subset \Sigma_Q$, then $\gamma_{j,i} = \gamma_{j,k} + \gamma_{k,i} \in \Sigma_Q$.
\end{proof}

\begin{defi}
	Let $Q$ be a subgroup of $G$ containing $H$. 
	We let  $\Sigma^{non}_Q$ be the set of $\chi \in \Sigma_Q$
	such that there exist $\chi_1, \chi_2 \in \Sigma_Q \setminus \{ \pm \chi\}$ such that
	$\chi =  \chi_1  +  \chi_2$.
\begin{comment}

	$\liealgg^{\chi}$
	is contained in the Lie algebra spanned by 
	\aryst
	\oplus_{\chi' \in \Sigma_Q  \setminus \{\pm\chi\}} \liealgg^{\chi'}.
	\earyst
	\end{comment}
\end{defi}

\begin{prop}\label{prop. liegroup2}
	Let $Q$ be a parabolic subgroup of $G$ with $|\Sigma \setminus \Sigma_Q| < 2n-2$.
	 Then for any subset $I \subset \Sigma \setminus \Sigma_Q$ such that $|I| < 2n-2  - |\Sigma| + |\Sigma_Q|$, there exists $\chi \in \Sigma \setminus (\Sigma_Q \cup I)$ such that
	 there exists $\chi' \in \Sigma_Q$  satisfying $\chi + \chi' \in \Sigma \setminus \Sigma_Q$. In particular, $-\chi \in \Sigma^{non}_Q$.
\end{prop}
\begin{proof}
	We first notice that if there exists $\chi' \in \Sigma_Q$ satisfying 
	\aryst
	\chi''' := \chi' + \chi \in \Sigma \setminus \Sigma_Q.
	\earyst
	Then as $Q$ is parabolic, $\chi'' := -\chi''' \in \Sigma_Q$. Consequently, we have
	$-\chi = \chi' + \chi''$. It is clear that $\chi', \chi'' \notin \{\pm \chi\}$. Thus $-\chi \in \Sigma^{non}_Q$.
	
	Without loss of generality, we may assume that $P \subset Q$. By \cite[V. 7, Proposition 5.90]{Knap} (see also  \cite[Section 2.1]{BRHW2}), there exist constants $1 \leq i_1 < \cdots < i_p \leq n-1$ such that
	\aryst
	\Sigma \setminus \Sigma_Q = \cup_{l=1}^{p} \{ \chi_{u,v} \mid v \leq i_p < u \}.
	\earyst
	The case where $n = 3$ can be verified directly. In the following we assume that $n \geq 4$.
	
	We first assume that there exists $1 \leq p \leq l$ such that $2 \leq i_p \leq n-2$. Then we have
	\aryst
	|\Sigma \setminus \Sigma_Q| \geq 2(n-2).
	\earyst
	Hence $|I| \leq 1$.
	
	We notice that both $\chi_{n,1}$, $\chi_{n,2}$ belongs to $\Sigma \setminus \Sigma_Q$. Thus there exists $\chi \in \{ \chi_{n,1}, \chi_{n,2}\}$ such that $\chi \in \Sigma \setminus (\Sigma_Q \cup I)$. Notice that $\chi_{n-1,n} \in \Sigma_Q$. Then we have
	\aryst
	\chi + \chi_{n-1,n} \in \{ \chi_{n-1, 1}, \chi_{n-1, 2}\} \subset \Sigma \setminus \Sigma_Q.
	\earyst
	This concludes the proof in this case.
	
	If there exists no such $\chi_p$, then there are only three possibilities for $\Sigma \setminus \Sigma_Q$ : 
	\enmt
	\item $\{n\} \times [n-1]$; 
	\item $([n]\setminus [1]) \times [1]$;
	\item the union of the above two.
	\eenmt
	In each of the above cases, we can verify the proposition directly.
\end{proof}

\begin{comment}

\begin{prop}\label{prop. liegroup3}
For any integer $1 \leq L < r(G)$ the following is true.

	Let $Q$ be a subgroup of $G$ such that $|\Sigma \setminus \Sigma_Q| \leq L $. Then for any subset $I \subset \Sigma \setminus \Sigma_Q$ such that $|I| \leq L  - |\Sigma| + |\Sigma_Q|$, there exists $\chi \in \Sigma \setminus (\Sigma_Q \cup I)$ such that the following is true:
	\enmt
	\item $-\chi \in \Sigma^{non}_Q$ and,
	\item there exist $\chi' \in \Sigma_Q$ and $\chi'' \in \Sigma \setminus \Sigma_Q$ such that
	\aryst
	\liealgg^{\chi''} \subset [ \liealgg^{\chi}, \liealgg^{\chi'}].
	\earyst
	\eenmt
\end{prop}

\end{comment}

%\begin{rema}	For $G = SL(n, \R)$, the statement of Proposition \ref{prop. liegroup2} can be simplified. We choose to state this version as it can be better adjusted to other simple complex Lie groups.
%\end{rema}

\section{Proof of the main theorem}

\subsection{Review of BFH}
The first step in the proof of Theorem \ref{thm: main2} is to show that $\a$ has uniform subexponential growth of derivatives.
\begin{defi}
	Let $\a : \Gamma \to \diff^{1}(M)$ be an action of $\Gamma$ on a compact manifold $M$ by $C^1$ diffeomorphisms. We fix an arbitrary $C^\infty$ Riemannian metric on $M$. We say that $\a$ has \textit{uniform subexponential growth of derivatives} if for every $\varepsilon > 0$ there is a constant  $C_\varepsilon > 0$ such that for all $\gamma \in \Gamma$ we have
	\aryst
	\norm{D\a(\gamma)} \leq C_{\varepsilon} e^{\varepsilon \ell( \gamma)}.
	\earyst
	It is clear that the above definition is independent of the choice of the metric on $M$.
\end{defi}
\begin{prop} \label{thm small derivates}
	Let $n \geq 3$ be an integer, and let $\Gamma < SL(n, \C) $ be a co-compact lattice. Let $M$ be a connected, compact manifold satisfying $\dim M < 2n-2$. Then $\a$ has uniform exponential growth of derivatives.
\end{prop}
We now start with the proof of Proposition \ref{thm small derivates}.

We assume to the contrary that $\a$ does not have uniform subexponential growth of derivatives.
Then by combining \cite[Proposition 3.6]{BFH1}, \cite[Claim 3.5]{BFH1} and \cite[Proof of Proposition 3.7]{BFH1}, we have
\begin{prop}\label{prop. prop3.7inbfn1}
	There exists an $s \in A$ and an $H$-invariant $H$-ergodic Borel probability measure $\mu$ on $M^\a$ with $\lambda^F_{+}(s, \mu) > 0$
	such that $\pi_*\mu$ is the Haar measure on $G / \Gamma$.
	Here $\lambda^F_{+}(s, \mu)$ is the \textit{maxmal fiberwise Lyapunov exponent} for $s \in A$ with respect to $\mu$ 
	given by the formula
	\aryst
	\lambda^F_+(s,\mu) = \inf_{n \to \infty}\frac{1}{n} \int \log \norm{D\ta(s^n) \restriction E_F(x)} d\mu(x).
	\earyst
\end{prop}

\subsection{From $H$ to $G$}\label{sec. from a to g}
To complete the proof of Proposition \ref{thm small derivates}, it remains to show the following.

\begin{prop}\label{prop: main}
	Assume that  $\dim M < 2n-2$.
	Let $\mu$ be an $H$-invariant $H$-ergodic measure on $M^{\a}$ such that $\pi_*\mu$ is the Haar measure on $G /\Gamma$. Then $\mu$ is $G$-invariant.
\end{prop}

The main technical proposition of our paper is the following.
\begin{prop}\label{main tech prop}
	Let $Q$ be a parabolic subgroup and let $\mu$ be an $Q$-invariant $H$-ergodic measure on $M^\a$. 
	Then for any $\chi \in \Sigma^{out}_3 \cap (-\Sigma^{non}_Q)$, the conditional measure $\mu^{G^{\chi}}_x$ is non-atomic for $\mu$-a.e. $x$.
\end{prop}
The proof of Proposition \ref{main tech prop} is divided into two  parts which occupy the next two sectons.

\begin{proof}[Proof of Proposition \ref{prop: main}]
	Assume that $\mu$ is not $G$-invariant. We set
	\aryst
	Q = \{g \in G \mid g_* \mu = \mu \}.
	\earyst
	By hypothesis, $H \subset Q \subsetneq G$.
	
	Define  $E^{\chi}_F, E^{\chi}_G, E^\chi, \Sigma_Q, \Sigma^{out}_1, \cdots$with respect to $\mu$.
	We claim that
	\aryst
	|\Sigma \setminus \Sigma_{Q}| \leq \dim M \leq 2n-3.
	\earyst
	Indeed, if this was not the case, then there would exist $\chi \in \Sigma \setminus \Sigma_{Q}$ such that $\chi$ is fiberwise non-resonant, i.e., $E^\chi_F = \{0\}$. By \cite[Proposition 5.1]{BRHW2}, we would deduce that $\mu$ is in fact $G^\chi$-invariant. This would contradict the definitions of $Q$ and $\Sigma_{Q}$.
	
	By Lemma \ref{prop. para}, we see that $Q$ is a parabolic subgroup. 
	
	We set
	\aryst
	I := ( \Sigma^{out}_1 \cup \Sigma^{out}_2 ) \setminus \Sigma_Q.
	\earyst
	By definition, \eqref{lab: sigmasigmaotusigmaq} and \eqref{lab: sigmaout12}, it is clear that $I \subset \Sigma \setminus \Sigma_Q$ and
	\aryst
	|I|  &=& |I| + |\Sigma^{out}_1 \cup \Sigma^{out}_2|  + |\Sigma^{out}_3| - |\Sigma^{out}| \\
	&\leq& 2|\Sigma^{out}_1 \cup \Sigma^{out}_2|  + |\Sigma^{out}_3| - |\Sigma^{out}| \\
	&\leq& \dim M - |\Sigma^{out}|  \\
	&\leq& \dim M - |\Sigma| + |\Sigma_Q| \\
	& < & 2n-2 - |\Sigma| + |\Sigma_Q|.
	\earyst
	By Proposition \ref{prop. liegroup2}, there exists
	\aryst
	\chi \in \Sigma \setminus (\Sigma_Q \cup I) \subset \Sigma^{out} \setminus ( \Sigma_Q \cup I ) \subset \Sigma^{out}_3
	\earyst
	such that there exists $\chi' \in \Sigma_Q$ satisfying
	\ary\label{lab. noncommutative2}
	\chi'' := \chi + \chi' \in \Sigma \setminus \Sigma_Q.
	\eary
	In particular, $-\chi \in \Sigma^{non}_Q$.

	By Proposition \ref{main tech prop}, the conditional measure $\mu^{G^{\chi}}_{x}$ is non-atomic for $\mu$-a.e. $x$.  By the $Q$-invariance of $\mu$, we see that $\mu^{G^{\chi'}}_x$ is Haar for $\mu$-a.e. $x$. Then by the method in \cite{EK, EK2} for noncommuting foliations along with \eqref{lab. noncommutative2}, we see that $\mu^{G^{\chi''}}_x$ is Haar for $\mu$-a.e. $x$. But this is a contradiction as this would imply that $\mu$ is $G^{\chi''}$-invariant, and consequently $\chi'' \in \Sigma_Q$.
\end{proof}

%\begin{proof}[Proof of Theorem \ref{thm: main}:]
%By \cite[Theorem 2.9]{BFH1} and \cite[Proposition 7]{DZ}, it is enough to show that $\a$ has uniform subexponential derivative growth. That is, we need to show that for every $\varepsilon > 0$ there is a $C_{\varepsilon}$ such that for all $\gamma \in \Gamma$ we have
%\aryst
%\norm{D\a(\gamma)} \leq C_\varepsilon e^{\varepsilon l(\gamma)}
%\earyst
%where $l: \Gamma \to \N$ denote the word-length function relative to some fixed finite set of generators.  

%By \cite[Claim 3.5]{BFH1}, we may assume to the contrary that the $G$-action $\ta$ on $M^\a$ fails to have uniform subexponential growth of fiberwise derivatives. Then by \cite[Proposition 3.6]{BFH1}, there is an $s \in A$ and an $A$-invariant measure $\mu$ on $M^\a$ with $\lambda^F_+(s,\mu) > 0$. From the proof of \cite[Proposition 3.7]{BFH1}, there is an $A$-invariant Borel probability measure $\mu_0$ on $M^\a$
%\end{proof}
\begin{proof}[Proof of Proposition \ref{thm small derivates}]
	Assume that $\a$ fails to have uniform subexponential growth of derivatives.
	 By Proposition \ref{prop. prop3.7inbfn1}, there is a $s \in A$ and an $H$-invariant $H$-ergodic  measure $\mu$ with
	$\lambda^F_{+}(s, \mu) > 0$, and  $\pi_*\mu$ is the  Haar measure on $G / \Gamma$. By Proposition \ref{prop: main}, we deduce that $\mu$ is $G$-invariant. We deduce that there exists a $\Gamma$-invariant measure $m$ on $M$.
	By Zimmer's cocycle superrigidity theorem, the $\Gamma$-action preserves a measurable metric on $M$. But in this case we should have $\lambda^F_{+}(s, \mu) = 0$. This is a contradiction.
	Thus $\a$ must has uniform subexponential growth of  derivatives.
\end{proof}

\begin{proof}[Proof of Theorem \ref{thm: main2} and \ref{thm: main3}:]
By Proposition \ref{thm small derivates}, we see that $\a$ has uniform subexponential growth of derivatives.
When $\a$ acts by $C^2$-diffeomorphisms, we conclude the proof of Theorem \ref{thm: main3} by the same argument in \cite{BFH1}.

By \cite[Theorem 2.9]{BFH1} and \cite[Proposition 7]{DZ}, we see that there exists a compact Lie group $K$; an injection $\iota : K \to {\rm Homeo}(M)$; and a group homomorphism $\phi : \Gamma \to K$ such that $\a = \iota\phi$. 

We  conclude the proof of Theorem \ref{thm: main2} by Margulis arithmetic theorem following \cite[Section 7]{BFH1}. Here we have used the fact that
\aryst
d(SL(n, \C)) &=& \begin{cases} 2n-2,  & n=3 \mbox{ or }n > 4, \\  5, & n = 4.  \end{cases} \\
v(SL(n, \C)) &=& 2n-2.
\earyst
\end{proof}

\smallskip
\

In the next two sections, we will give the proof of Proposition \ref{main tech prop}. We let $Q$ be a parabolic subgroup of $G$, and let $\mu$ be a $Q$-invariant $H$-ergodic measure; and let $\chi \in \Sigma^{out}_3  \cap (-\Sigma^{non}_Q)$. 
We also denote by $\chi_F \in \chi$ the Lyapunov functional for $E^{\chi}_F$, and denote by $\chi_G \in \chi$ the Lyapunov functional for $E^{\chi}_G$. 
To simply notation, we denote $E^{\chi}_F$ by $E$. By our hypothesis that $\chi \in \Sigma^{out}_3$, we have $\dim E = 1$.

\section{When $\mu^{\cW^\chi_F}$ is non-atomic}
\label{sec. nonatomic}

Through out this section, we assume that for $\mu$-a.e. $x$, the support of $\mu^{\cW^{\chi}_{F}}_x$ is non-discrete with respect to the leafwise metric.

\subsection{Time change and measurable Lyapunov foliation}
We fix a small constant $\varepsilon > 0$. 
 As in \cite[Section 5]{KKRH}, for any Lyapunov regular point $x \in M^{\a}$, for any $u,v \in E(x)$,  we define {\it the standard $\varepsilon$-Lyapunov scalar product} 
\aryst
\langle  u, v \rangle_{\varepsilon} = \int_{\liealga} \langle D\ta(s)u, D\ta(s)v \rangle  \exp(-2\chi(s) -2\varepsilon \norm{s})ds.
\earyst
For any $C > 0$, we define the Pesin set
$R(C)$ as in \cite[Proposition 2.2]{KKRH}. 
By \cite[Remark below Proposition 5.3]{KKRH}, we have
\ary \label{item: pesinssetandmap}
\ta(s) R(C) \subset R(e^{2\norm{s}\varepsilon}C), \quad \forall s \in \liealga.
\eary
We summarize the time change argument in \cite{KKRH} (more specifically, Proposition 6.2-6.7 in \cite{KKRH}) in the following two lemmata.
As in \cite{KKRH},
we fix an element $w \in \liealga$ such that $\chi(w) =1$.

\begin{lemma}\label{lem time change}
For $\mu$-a.e. $x \in M^{\a}$ and any $t \in \liealga$ there exists a real number $g(x,t)$ such that the function $\bg(x,t) = t + g(x,t)w$ satisfies the following property. The measurable map
\aryst
\tb(t,x) = \ta(\bg(x,t))x
\earyst
is an $\liealga$-action preserving a probability measure $\tilde{\mu}$ which is absolutely continuous with respect to $\mu$ with positive density,
and for any $t \in \liealga$ we have
\aryst
\norm{D \a(\bg(x,t))|_{E(x)}}_{\varepsilon} = e^{\chi(t)}.
\earyst
The function $g(x,t)$ is measurable and is continuous in $x$ on Pesin set and along the orbits of $\ta$. Moreover, $\bg(x,t)$ is $C^1$ in $t$ and it satisfies that 
\ary \label{item: g}
|g(x,t)| \leq 2\varepsilon \norm{t}, \quad |\partial_t g(x,t)|		 \leq \varepsilon.
\eary
\end{lemma}

\begin{lemma}\label{lem fake stable manifold}
For any $s \in \liealga$ there is a stable \lq\lq foliation\rq\rq $\tcW^{-}_{\tb(s)}$ which is contracted by $\tb(s)$ and invariant under the new action $\tb$. It consists of \lq\lq leaves\rq\rq  $\tcW^{-}_{\tb(s)}(x)$ defined for $\mu$-a.e. $x$. The \lq\lq leaf \rq\rq $\tcW^-_{\tb(s)}(x)$  is a measurable subset of the leaf $\ta(\R w)\cW^{-}_{\ta(s)}(x)$ of the form
\aryst
\tcW^{-}_{\tb(s)}(x) = \{  \ta(\varphi^{s}_x(y)w)y \mid y \in \cW^{-}_{\ta(s)}(x) \}
\earyst
where $\varphi^{s}_x: \cW^{-}_{\ta(s)}(x) \to \R$ is a $\mu^{\cW^{-}_{\ta(s)}}_x$-almost everywhere defined measurable function. For $x$ in a Pesin set, $\varphi^s_x$ is H\"older continuous on the intersection of this Pesin set with any ball of fixed radius on $\cW^{-}_{\ta(s)}(x)$ with H\"older exponent $\gamma$ and H\"older constant which depends on the Pesin set and radius.
\end{lemma}

We have the following observation.
\begin{lemma}\label{lem. worbitboundeddistortion}
For $\mu$-a.e. $x$, for any $t \in \R$, for any $k \in \liealga$, we have $\tb(k)\ta(tw)x = \ta(sw) \tb(k)x$ where $s \in \R$ satisfies
\aryst
 |t|/4 < |s| < |t|.
\earyst
\end{lemma}
\begin{proof}
	For $\mu$-a.e. $y$, we define function $\phi_y : \R \to \R$ by
	\aryst
	\phi_y(t) = t + g(y,tw), \quad \forall t \in \R.
	\earyst
	Then it is clear that for $\mu$-a.e. $y$,
	\aryst
	\tb(tw)y = \ta(\phi_y(t)w)y.
	\earyst
	By \eqref{item: g}, we see that for $\mu$-a.e. $y$, $\phi_y$ is a diffeomorphism of $\R$ with $\norm{\phi_y}, \norm{\phi_y^{-1}} < 2$.
	
	By our choices of $s,t$, we have
	\aryst
	\tb(k) \ta(tw) x = \tb(k) \tb(\phi_x^{-1}(t) w)x
	= \ta(sw) \tb(k)x 
	= \tb(\phi_{\tb(k)x}^{-1}(s)w) \tb(k)x.
	\earyst
	Consequently, we have $s = \phi_{\tb(k)x}\phi_x^{-1}(t)$. This concludes the proof.
\end{proof}

In \cite[Corollary 6.8]{KKRH}, the authors gave the existence of coarse Lyapunov foliations for $\tb$.
In the following, we give a detailed account.
\begin{lemma}\label{lem fake coarse leaf}
For any $\chi' \in \Sigma$, 
there exists an essentially unique\footnote{We say that two collections $\{\varphi_x : \cW^{\chi'}(x) \to \R \}_{x \in M^\a}$ and  $\{\phi_x : \cW^{\chi'}(x) \to \R \}_{x \in M^\a}$ are equivalent if for $\mu$-a.e. $x$, for $\mu^{\cW^{\chi'}}_x$-a.e. $y$, we have $\varphi_x(y) = \phi_x(y)$.} collection $\varphi^{\chi'} = \{\varphi^{\chi'}_x : \cW^{\chi'}(x) \to \R \}_{x \in M^\a}$ where for $\mu$-a.e. $x$, $\varphi^{\chi'}_x$ is a $\mu^{\cW^{\chi'}}_x$-a.e. defined function and H\"older continuous on Pesin sets, such that the following is true:
for any $k \in \liealga$, for $\mu$-a.e. $x$, for $\mu^{\cW^{\chi'}}_x$-a.e. $y$, set $z= \ta(\varphi^{\chi'}_x(y)w)y$, we have
\aryst
\tb(k,z) = \ta(\varphi^{\chi'}_{\tb(k,x)}(\ta(\bg(x, k))y)w)\ta(\bg(x, k))y.
\earyst
We have a similar collection of measurable functions for $\tcW^{\chi}_F$.
\end{lemma}
\begin{proof}
We claim that for any $s,k \in \liealga$, for $\mu$-a.e. $x$, for $\mu^{\cW^-_{\ta(s)}}_x$-a.e. $y$, set $z = \ta(\varphi^s_x(y)w)y$, we have
\ary \label{item: claim1}
\tb(k, z) =  \ta(\varphi^s_{\tb(k,x)}(\ta(\bg(x, k))y)w)\ta(\bg(x, k))y.
\eary
To prove the claim, we first notice that by definition we have
\aryst
\tb(k,z) &=& \ta( \bg(z,k))z \\
&=& \ta( \bg(z,k))\ta(\varphi^s_x(y)w)y \\
&=& \ta( \bg(z,k))\ta(\varphi^s_x(y)w) \ta(-\bg(x,k))  \ta(\bg(x,k))y \\
&=& \ta( (g(z,k)-g(x,k) + \varphi^s_x(y))w ) \ta(\bg(x,k))y \\
&=& \ta(t'w) y'
\earyst
where 
\aryst
t' = g(z,k)-g(x,k) + \varphi^s_x(y), \quad 
y' = \ta(\bg(x,k))y.
\earyst
Notice that we have $y' \in \cW^-_{\ta(s)}(\tb(k,x))$.

By Lemma \ref{lem fake stable manifold} and the fact that $\tcW^-_{\tb(s)}$ is $\tb(\liealga)$-invariant, we see that $z \in \tcW^{-}_{\tb(s)}(x)$ and $\tb(k,z) \in \tcW^-_{\tb(s)}(\tb(k,x))$. Thus there exists $y'' \in \cW^-_{\ta(s)}(\tb(k,x))$ such that 
\aryst
\tb(k,z) = \ta(\varphi^s_{\tb(k,x)}(y'')w)y''.
\earyst
Consequently, $y'' \in \cW^{-}_{\ta(s)}(y')$, and  there exists $t'' \in \R$ such that 
\aryst
y'' = \ta(t'' w) y'.
\earyst
Since we have $d(\ta(ns)y', \ta(ns)y'') \to 0$ as $n$ tends to infinity, we can show that $t'' = 0$ for a $\mu$-typical $y$. Consequently, $y'' = y'$. This proves our claim.

Fix an arbitrary $s \in \liealga$ such that 
\ary \label{item: chi's}
\chi'(s) < -10\norm{s} \varepsilon.
\eary
Then by Lemma \ref{lem fake stable manifold}, for $\mu$-a.e. $x$, function $\varphi^s_x$ is defined $\mu^{\cW^-_{\ta(s)}}_x$ almost everywhere. Thus for $\mu^{\cW^-_{\ta(s)}}_x$-a.e. $y$, $\varphi^s_x(z)$ is defined for $\mu^{\cW^{\chi'}}_y$-a.e. $z$. We define $\varphi^{\chi'}_x$ to be the restriction of $\varphi^s_x$ to $\cW^{\chi'}(x)$ for $\mu$-a.e. $x$. In the following we abbreviate $\varphi^{\chi}_x$ as $\varphi_x$.

We also show that $\varphi_x$ defined above is essentially independent of the choice of $s$. Take another $s' \in \liealga$ with $\chi'(s') < 0$. Assume that there exists a set $\Omega \subset M^\a$ with $\mu(\Omega) > 0$ such that for every $x \in \Omega$, there exists a subset  $\Omega_x \subset \cW^{\chi'}(x)$ with  positive $\mu^{\cW^{\chi'}}_x$ measure such that for every $y \in \Omega_x$, $\varphi_x(y) = t''w + \varphi^{s'}_x(y)$ for some $t'' \neq 0$. On the other hand, by \eqref{item: pesinssetandmap}, \eqref{item: chi's}, \eqref{item: claim1} and by the H\"older continuity of $\varphi^{s'}_x$, $\varphi_x$ on Pesin sets,  we see that for typical choices of $x,y$, we have
\aryst
d(\tb(ns)y', \tb(ns)y'') \to 0 \mbox{ as } n \to \infty
\earyst
where $y'=\ta(\varphi_x(y)w)y, y''= \ta(\varphi^{s'}_x(y)w)y$.
By Lemma \ref{lem. worbitboundeddistortion}, this contradicts $t'' \neq 0$. Consequently, we see that the definition of $\varphi_x$ is independent of the choice of $s$. This concludes the proof.
\end{proof}
From the proof of Lemma \ref{lem fake coarse leaf}, we can deduce the following.
\begin{cor}\label{cor: graph coincides}
For any $\chi' \in \Sigma$, for any $s \in \liealga$ such that $\chi'(s) < 0$, for $\mu$-a.e. $x$, for $\mu^{\cW^{\chi'}}_x$-a.e. $y$, we have
\aryst
\varphi^{s}_x(y) = \varphi^{\chi'}_x(y) 
\earyst
where $\varphi^{s}_x$ is given by Lemma \ref{lem fake stable manifold}, and $\varphi^{\chi'}_x$ is given by Lemma \ref{lem fake coarse leaf}.
\end{cor}

%Given an arbitrary $\chi' \in \Sigma$. Let $\{ \varphi_x \}_{x \in M^\a}$ be given as in Corollary \ref{cor: graph coincides}.
%As in \cite[Lemma 6.16]{KKRH}, we can define a measurable partition $\xi$ subordinate to $\tcW^{\chi'}$, i.e., for $\mu$-a.e. $x$, there is an open neighborhood $U_x \subset \cW^{\chi'}(x)$ such that $\{ \ta(\varphi_x(y)w)y \mid y \in U_x \} \subset \xi(x)$. For any $k_0 \in \liealga$ with $\chi'(k_0) > 0$, and denote $f= \ta(k_0)$. We can also ensure that 
%\enmt
%\item $\xi$ is $f$-increasing,
%\item $\vee_{i=0}^\infty f^{-i}\xi$ is the partition into points,
%\item $\cap_{i=0}^\infty \cB_{f^{i} \xi} = [\tcW^{\chi'}]$.
%\eenmt

By Lemma \ref{lem fake coarse leaf} and Corollary \ref{cor: graph coincides}, we can define for each $\chi' \in \Sigma$ a collection of $\tb(\liealga)$-invariant sets $\tcW^{\chi'}$ by setting
\aryst
\tcW^{\chi'}(x) = \{ \ta(\varphi_x(y)w)y \mid y \in \cW^{\chi'}(x) \}
\earyst 
where $\varphi_x$ is given by Lemma \ref{lem fake coarse leaf} associated to $\chi'$. Similarly, we can define $\tcW^{+}_{\tb(a)}$, $\tcW^{\chi'}_F$ and $\tcW^{\chi'}_G$.

We have the following useful lemma.
\begin{lemma}\label{cor: conditional measure fake leaf}
For any $\chi' \in \Sigma$, for $\mu$-a.e. $x$, the conditional measure
$\tilde{\mu}^{\tcW^{\chi'}}_x$  is absolutely continuous with respect to $(y \mapsto \ta(\varphi_x(y)w)y)_* \mu^{\cW^{\chi'}}_x$. We have analogous statements for $\cW^{\chi'}_F$, $\cW^{\chi'}_G$ and $\cW^-_{\ta(a)}$, $a \in \liealga$.
\end{lemma}
\begin{proof}
This is proved in the last paragraph of \cite[Lemma 7.1]{KKRH}.
\end{proof}

\begin{rema} \label{rema.productmeasure}
The proof of Lemma \ref{cor: conditional measure fake leaf} uses the following fact. For any $\chi' \in \Sigma$, for $\mu$-a.e. $x$, we have
\aryst
\ta(\R w) \tcW^{\chi'}_*(x) = \ta(\R w) \cW^{\chi'}_*(x), * = \emptyset, F, G.
\earyst
In particular, when $\chi' \in \Sigma_Q$, the conditional measure of $\tilde \mu$ on $\ta(\R w) \tcW^{\chi'}_G(x)$ is equivalent to the natural push-forward of the Lebesgue measure on $\R \times G^{\chi'}$.
\end{rema}
By Remark \ref{rema.productmeasure}, for every $\chi' \in \Sigma_Q$ we can define $\tcW^{\chi'}_G$-holonomy maps between $\ta(\R w)$-orbits within   $\ta(\R w) \cW^{\chi'}_G(x)$ for $\mu$-a.e. $x$.

\begin{lemma}\label{lem. fake foliation is ac}
	Let $x$ be a  $\mu$-typical point, and let $h \in G^{\chi'}$ such that the $\tcW^{\chi'}_G$-holonomy map between $\ta(\R w)x$ and $\ta(\R w)\ta(h)x$ is defined for Lebesgue almost every point. Then
	the $\tcW^{\chi'}_G$-holonomy map between $\ta(\R w)x$ and $\ta(\R w)\ta(h)x$ is absolutely continuous.
\end{lemma}
\begin{proof}
	We will show that this $\tcW^{\chi'}_G$-holonomy map extends to a Lipschitz map.
	
	For $i=1,2$, we take $t_i, s_i \in \R$ such that $\tcW^{\chi'}_G(\ta(t_iw)x)$ intersects $\ta(\R w)\ta(h)x$ at $\ta(s_i w)\ta(h)x$.
		 Take an arbitrary $a \in \liealga$ such that $\chi'(a) < 0$.
	For any integer $n > 0$, we denote by $u_n, v_n \in \R$ constants such that
	\aryst
	\tb(n a)\ta(t_1 w)x &=& \ta(u_n w)\tb(na)\ta(t_2w)x, \\
	\tb(na) \ta(s_1 w) \ta(h)x &=& \ta(v_n w) \tb(na) \ta(s_2w)\ta(h)x.
	\earyst
	On one hand, we know that for $i=1,2$, \aryst
	d(\tb(na) \ta(t_i w)x, \tb(na) \ta(s_i w)\ta(h)x) \to 0 \quad \mbox{ as } \quad n \to +\infty.
	\earyst
 This implies that $|v_n-u_n|$ tends to $0$ as $n$ tends to infinity.
 On the other hand, by Lemma \ref{lem. worbitboundeddistortion}, we see that both $|\frac{t_1-t_2}{u_n}|$, $|\frac{s_1 - s_2}{v_n}|$ are bounded from above and from below by constants  independent of $n$.  This implies our lemma.
\end{proof}

\subsection{The proof for the non-atomic case}

\begin{proof}[Proof of Proposition \ref{main tech prop} --- the non-atomic case]
Our argument is an adaptation of the $\pi$-partition trick (see \cite{KK, KKRH}).
The main tool is the following lemma.
\begin{lemma}
For any Pesin set $R$, there exists a constant $K > 0$ such that for $\mu$-a.e. $x \in R$, and $\mu^{\cW^{\chi}_F}_x$-a.e. $y \in R \cap \cW^{\chi}_{F, loc}(x)$, there exists a sequence $\{l_n\}_{n \in \N} \subset \liealga$ satisfying that $\ta(l_n)x \to y \mbox{ as } n \to \infty$ and $\norm{D\ta(l_n)|_{E(x)}} < K$.
\end{lemma}
\begin{proof}
 
In the following,  for any $b \in \liealga$, we let $[\tcW^{-}_{\tb(b)}]$ denote the sub-$\sigma$-algebra of $\cB_{M^{\a}}$ whose elements are unions of subsets of the form $\tcW^{-}_{\tb(b)}(y)$ with $y \in M^{\a}$ (this was previously introduced in  \cite[Section 6.4]{KKRH}). We define $[\tcW^{\chi}]$ and $ [\tcW^{\chi}_F] $ analogously. 
For any $a \in \liealga$, we denote by $\tcE_{\tb(a)}$ the sub-$\sigma$-algebra of $\cB_{M^\a}$ formed by $\tb(a)$-invariant sets.

We take a singular generic $a \in L_\chi$, i.e., $a \in L_\chi$ but $a \notin L_{\chi'}, \forall \chi' \in \Sigma \setminus \{\pm \chi \}$, and some generic $b \in \liealga$, close to $a$, such that $\chi(b)  > 0$ and $\chi'(b)$, $\chi'(a)$ have the same sign for all $\chi' \in \Sigma \setminus \{\pm \chi \}$. Then by \cite{LedYou1}, we have
\ary  \label{item: chifhchi}
 [\tcW^{\chi}_F] \supset [\tcW^{\chi}] \supset  [\tcW^{+}_{\tb(b)}] = [\tcW^{-}_{\tb(b)}].
 \eary

We also have the following inclusion.
\begin{lemma}\label{lem: tbbtbaminuschi}
We have $ [\tcW^{-}_{\tb(b)}] \supset  [\tcW^{-}_{\tb(a)}] \cap [\tcW^{-\chi}]$.
\end{lemma}
\begin{proof}
Take an arbitrary function $f \in L^2(M^\a, \tilde{\mu})$. We set 
\aryst
f_0 = \mathbb{E}(f \mid [\tcW^{-}_{\tb(a)}] \cap [\tcW^{-\chi}]).
\earyst
Then by definition, Lemma \ref{cor: conditional measure fake leaf} and Corollary \ref{cor: graph coincides}, there exist $\mu$-conull sets $\Omega_0, \Omega_1 \subset M^{\a}$ such that:
\enmt
\item for every $x \in \Omega_1$, for $\mu^{\cW^{-\chi}}_x$-a.e. $y$,  the point  $z := \ta( \varphi^\chi_x(y)w)y$ satisfies $f_0(x) = f_0(z)$, and $\varphi^\chi_x(y) = \varphi^b_x(y)$;

\item for every $x \in \Omega_0$, $\mu^{\cW^{-}_{\ta(a)}}_x$-a.e. $y$ belongs to $\Omega_1$. Moreover, for every $y \in \Omega_1 \cap \cW^{-}_{\ta(a)}(x)$, we have $z := \ta( \varphi^{a}_x(y)w)y \in \Omega_1$, $f_0(x) = f_0(z)$, and  $\varphi^{a}_x(y) = \varphi^b_x(y)$.
\eenmt

We claim that for $\mu$-a.e. $x$, for $\tilde{\mu}^{\tcW^{-}_{\tb(b)}}_x$-a.e. $z$, we have $f_0(x) = f_0(z)$. This will imply that $f_0$ is $[\tcW^{-}_{\tb(b)}]$-measurable, and conclude the proof.

As $a \in L_{\chi}$,  we have $\ta(a) \ta(g) = \ta(g) \ta(a)$ for any $g \in G^{-\chi}$.
Thus for any $x \in M^{\a}$, any $y \in \cW^{-}_{\ta(a)}(x)$, any $g \in G^{-\chi}$, we have $\ta(g)y \in \cW^-_{\ta(a)}(\ta(g)x)$. By $\chi \in \Sigma^{out}_3$, we know that $\mu$ is $G^{-\chi}$-invariant.  This implies that for any $g \in G^{-\chi}$, for $\mu$-a.e. $x$, we have
\aryst
[(\ta(g))_* \mu^{\cW^{-}_{\ta(a)}}_x] = [\mu^{\cW^{-}_{\ta(a)}}_{\ta(g)x}].
\earyst
Thus for $\mu$-a.e. $x \in \Omega_0$, for a $\mu^{\cW^-_{\ta(b)}}_x$-typical $y$, there exists $y' \in \Omega_1 \cap \cW^{-}_{\ta(a)}(x)$ such that $y \in \cW^{-\chi}(y')$. Set $z' = \ta(\varphi_x^{b}(y')w)y'$ and $z''= \ta(\varphi_x^{b}(y')w)y \in \cW^{-\chi}(z')$.
Then $z' \in \Omega_1$ and $f_0(x) = f_0(z')$. As the holonomy map between $\cW^{-\chi}(y')$ and $\cW^{-\chi}(z')$ along $\ta(\R w)$-orbits is absolutely continuous, a typical choice of $y \in \cW^{-\chi}(y')$ corresponds to a typical choice of $z'' \in \cW^{-\chi}(z')$. Thus for a typical $y$ we have
\aryst
z  :=  \ta(\varphi^b_{x}(y)w)y = \ta(\varphi^\chi_{z'}(z'')w)z''.
\earyst 
Consequently, we have $f_0(z) = f_0(z') = f_0(x)$. 
\end{proof}

\begin{comment}
Need \cite[Prop 8.2.3]{Springer}
\end{comment}

\begin{lemma}\label{lem.accforfakefoliation}
We have $[\tcE_{\tb(a)}] \subset [\tcW^{-\chi}]$.
\end{lemma}
\begin{proof}
	 By $\chi \in \Sigma^{out}_3$, we have $\tcW^{-\chi}_G = \tcW^{-\chi}$. Thus it suffices to show that $[\tcE_{\tb(a)}] \subset [\tcW^{-\chi}_G]$.
 
   We fix a continuous function $\theta$ on $M^\a$. We define
   \aryst
   B^{\pm}_\theta = \{ x \mid \lim_{n \to \pm \infty} \frac{1}{n}\sum_{i=0}^{n-1} \theta(\tb(na)x) = \int
    \theta d\tilde \mu^{\tcE_{\tb(a)}}_x \}
   \earyst
   where $\mu^{\tcE_{\tb(a)}}_x$ denotes the $\tb(a)$-ergodic component of $\tilde \mu$ at $x$.
   
   By Birkhoff's ergodic theorem, we know that for $\tilde \mu$-a.e. $x$, $B^+_\theta(x) = B^-_\theta(x)$. In this case, we say that $x$ is regular (with respect to $\theta$) and denote $B_\theta(x) := B^{\pm}_\theta(x)$.
   Consequently, by the $\ta(\R w)$-invariance of $\mu$ and the fact that $\tilde \mu \sim \mu$, the conditional measures of $\tilde \mu$ along $\ta(\R w)$-orbits are absolutely continuous with respect to Lebesgue. Thus for $\tilde\mu$-a.e. $x$, for Lebesgue almost every $t\in \R$,  $\ta(tw) x$ is regular.
   
   We let $W$ be the set of $x \in M^\a$ such that for $\eta \in \{ -\chi, \chi_1, \chi_2\}$, $x$ satisfies Corollary \ref{cor: graph coincides}. We know that $W$ is a $\tilde \mu$-conull set. Then for $\tilde \mu$-a.e. $x$, for Lebesgue almost every $t\in \R$,  $\ta(tw)x \in W$.
   
   By Fubini's lemma, Remark \ref*{rema.productmeasure} and the above discussion on regular points, we know that for $\eta \in \{ -\chi, \chi_1, \chi_2\}$, for $\nu_{G^\eta}$-a.e. $h$, for $\tilde\mu$-a.e. $x$, for Lebesgue almost every $t\in \R$, $\ta(tw)x$ is regular and $\varphi^\eta_{\ta(tw)x}$ is defined at $\ta(tw) \ta(h)x$.  We denote the above $\nu_{G^\eta}$-conull set by $\Omega_{\eta}$, and for every $h \in \Omega_\eta$ we denote by $W_h$ the above $\mu$-conull set of $x$. 
   
       By $\chi \in -\Sigma^{non}_Q$, there exist $\chi_1, \chi_2 \in \Sigma_Q \setminus \{\pm \chi \} $ such that $-\chi = \chi_1 + \chi_2$.
   Then for $\nu_{G^{-\chi}}$-a.e. $h$, there exist $h_i, h_{i+2} \in \Omega_{\chi_i}$, $i=1,2$ such that $h = h_4 h_3 h_2h_1$. 
   
   It is direct to see that $\tilde \mu$-a.e. $x$ satisfies that 
   \aryst
   x \in W_{h_1},  \quad
   \ta(h_1)x \in W_{h_2}, \quad \ta(h_2h_1)x \in W_{h_3}, \quad
   \ta(h_3h_2h_1)x \in W_{h_4}.
   \earyst
   
   By Lemma \ref{lem. fake foliation is ac}, the $\tcW^{\chi_1}_G$-holonomy map between $\ta(\R w)x$ and $\ta(\R w)\ta(h_1)x$ within the leaf $\ta(\R w)\cW^{\chi_1}_G(x)$ is absolutely continuous.
   More precisely, for Lebesgue almost every $t \in \R$, the intersection between $\tcW^{\chi_1}_G(\ta(tw)x)$ and $\ta(\R w)\ta(h_1)x$ is $\ta(\phi(t)w ) \ta(h)x$
   where 
   \aryst
   \phi(t) = \varphi^{\chi_1}_{\ta(tw)x}(\ta(tw)\ta(h_1) x)+t.
   \earyst
   Lemma \ref{lem. fake foliation is ac} implies that $\phi$ preserves the Lebesgue class.
    Consequently, for Lebesgue almost every $t\in\R$, 
   $\ta(tw)x, \ta(\phi(t)w ) \ta(h) x$ are both regular.
   
   By iterating the above argument, we see that for Lebesgue almost every $t \in \R$, there exist regular points $x_1,\cdots,x_4$ such that the following is true. Set $x_0 = \ta(tw) x$. We have
   \aryst
   x_1 \in \tcW^{\chi_1}_G(x_0), 
   x_2 \in \tcW^{\chi_2}_G(x_1),
   x_3 \in \tcW^{\chi_1}_G(x_2), 
   x_4 \in \tcW^{\chi_2}_G(x_3).
   \earyst
   Moreover, there exists $s \in \R$ such that $x_4 = \ta(s w) \ta(h)x_0$.
   
   By definition, it is easy to see that
   \aryst
   B_{\theta}(x_0) = B_{\theta}(x_4),
   \earyst
   and for some $c \in \liealga$ such that both $\chi_1(c), \chi_2(c) < 0$, we have
   \aryst
   d(\tb(nc)x_0, \tb(nc)x_4) \to 0 \quad \mbox{as} \quad n \to \infty.
   \earyst
   This implies that $x_4 \in \tcW^{-\chi}_G(x_0)$ and consequently $s = \varphi^{-\chi}_{x_0}(\ta(h)x_0)$.
   
   Finally, by Fubini's lemma, we deduce that for $\mu$-a.e. $x$, for $\tcW^{-\chi}_G$-a.e. $y$, we have $B_\theta(x) = B_\theta(y)$. By take $\theta$ over a dense subset of $L^1(M^\a, \tilde \mu)$, we conclude the proof.
\end{proof}

It is well-known that 
\aryst
[\tcE_{\tb(a)}] \subset [\tcW^{-}_{\tb(a)}].
\earyst
Consequently, by Lemma \ref{lem: tbbtbaminuschi} and Lemma \ref{lem.accforfakefoliation} we have
\aryst
 [\tcE_{\tb(a)}] \subset [\tcW^{\chi}_F].
\earyst

\begin{comment}
By definition and by \cite{LedYou1}, we have 
\ary \label{item.inclu3}
[\tcW^{-}_{\tb(a)}] \supset \tcE_{\tb(a)} \supset \tcE_{H_{\chi}}.
\eary
Also, we notice that $\tcW^{-\chi} = \tcW^{-\chi}_{G}$ by $\chi \in \Sigma^{out}_3$.

Recall that $-\chi \in \Sigma^{non}_Q$.
Then by \cite[ Chapter I, Section 2, Proposition 2.2]{Bor} or \cite[Section 7.5]{Hum}, for any $h \in G^{-\chi}$, there exist an integer $k \geq 1$; $\chi_1, \cdots, \chi_k \in \Sigma_Q \setminus \{\pm \chi\}$; and for every $1 \leq i \leq k$, an element $h_i \in G^{\chi_i}$, such that 
\ary \label{item. hdecompose}
h = h_k \cdots h_1.
\eary

For each $1 \leq i \leq k$, we choose $a_i \in \liealga$ such that $\chi_i(a_i) < 0$.
By definition, for each $1 \leq i \leq k$ we have
\ary \label{item. inclusion2}
[ \tcW^{\chi_i}_G] \supset [ \tcW^{-}_{\tb(a_i)}] \supset [\tcE_{\tb(a_i)}] \supset [\tcE_{H_\chi}].
\eary
The following claim can be deduced from \eqref{item. hdecompose} and a similar argument from Lemma \ref{lem: tbbtbaminuschi}.
\begin{claim} \label{claim. inclu}
	We have $[\tcW^{-\chi}] = [\tcW^{-\chi}_{G}] \supset \bigcap_{1 \leq i \leq k} [\tcW^{\chi_i}_{G}]$.
\end{claim}
By Claim \ref{claim. inclu} and \eqref{item. inclusion2}, we have
\aryst
[\tcW^{-\chi}] \supset \tcE_{H_{\chi}}.
\earyst
Then by Lemma \ref{lem: tbbtbaminuschi}, \eqref{item.inclu3} and \eqref{item: chifhchi}, we have 
\aryst
 [\tcW^{\chi}_F] \supset \tcE_{H_{\chi}}.
\earyst

\end{comment}

Let $R$ be the Pesin set in the lemma. 
	Let $\varphi_x$ be given by Lemma \ref{lem fake coarse leaf} for $\tcW^{\chi}_F$.
By Lemma \ref{lem fake coarse leaf}, there exists $K_1 > 0$ such that 
$|\varphi_x(y)| < K_1$ for any $y \in R \cap \cW^{\chi}_{F, loc}(x)$. Then by \eqref{item: pesinssetandmap}  the point $z = \ta(\varphi_x(y)w)y$ belongs to a Pesin set $R' \supset R$ which depends on $R$, but is independent of $x$ and $y$.

By Lemma \ref{cor: conditional measure fake leaf}, for $\mu$-a.e. $x$, for $\mu^{\cW^{\chi}_F}_x$-a.e. $y \in R \cap \cW^{\chi}_{F, loc}(x)$, $\ta(\varphi_x(y)w)y$ is a $\tilde{\mu}^{\tcE_{\tb(a)}}_x$- density point of $R'$.
  Then by Birkhoff's ergodic theorem, for the above $x,y,z$ there exists a sequence $\{ k_n \}_{n \in \N} \subset \N$ such that $\tb(k_n a)x \in R'$ converges to $z$ as $n$ goes to infinity. Let $l_n = \bg(x,k_n a)-\varphi_x(y)w$, then we have
\aryst
\ta(l_n)x = 
\ta(-\varphi_x(y)w)\tb(k_n a)x \to y \quad\mbox{ as }\quad n \to \infty.
\earyst
We have
\aryst
D\ta(l_n)(x) = D\ta(-\varphi_x(y)w)(\tb(k_n a)x) D\ta(\bg(x,k_n a))(x).
\earyst
Moreover
\aryst
\norm{D\ta(\bg(x, k_n a))|_{E(x)}} \leq K_2^{2}
\earyst
where $K_2$ is the maximal distortion between $\norm{\cdot}_{\varepsilon}$ and the background metric on $M^{\a}$ over the Pesin set $R'$. By $|\varphi_x(y)| < K_1$, we can see that there exists $K > 0$ depending only on $R$, such that 
\aryst
\norm{D\ta(l_n)|_{E(x)}} < K.
\earyst
This concludes the proof.
\end{proof}
By the argument in \cite{KK}, we see that $\mu^{\cW^{\chi}_F}_x$ is absolutely continuous with positive density Lebesgue almost everywhere.  As $\cW^{\chi}$ is $C^1$ foliated by $\cW^{\chi}_F$ and $\cW^\chi_G$, by the absolute continuity of the $\cW^\chi_G$-holonomy maps between different $\cW^{\chi}_F$-leaves, we deduce that $\mu^{\cW^{\chi}_G}_x$ is non-atomic for $\mu$-a.e. $x$.  Consequently, $\mu^{G^{\chi}}_x$ is non-atomic for $\mu$-a.e. $x$.
\end{proof}

\section{When $\mu^{\cW^\chi_F}$ is atomic}
\label{sec.atomic}

Through out this section, we assume that for $\mu$-a.e. $x$, $\mu^{\cW^{\chi}_{F}}_x$ is supported on a discrete set with respect to the leafwise metric. Then the following result is well-known.
\begin{lemma}
	For $\mu$-a.e. $x$, $\mu^{\cW^\chi_F}_x$ is the Dirac measure at $x$.
\end{lemma}
\begin{proof}
	Assume to the contrary that the lemma fails.
	For $\mu$-a.e. $x$, we define 
	\aryst
	r(x) := \sup \{ \sigma \mid \mu^{\cW^\chi_F}_x(B(x,\sigma) \setminus \{x\}) = 0 \}.
	\earyst
	We obtain a contradiction by the $A$-invariance of $\mu$ and Poincar\'e's recurrence theorem. 
\end{proof}

\subsection{A local entropy forumula}\label{sec: A local entropy forumula}

In this subsection, we recall a local entropy formula from \cite{LedYou2}.

We fix an arbitrary $k \in \liealga$ such that $\chi(k) > 0$.  Let us denote $f= \ta(k)$.

By the construction in \cite[Section 3]{LedYou1}, we can also choose two measurable partitions $\eta_0$ and $\eta_1$ such that
\enmt
\item $\eta_0$, resp. $\eta_1$, is subordinate to $\cW^{\chi}_G$, resp. $\cW^\chi$;
\item $\eta_0$, $\eta_1$ are all $f$-increasing and $f$-generating;
\item $\eta_0 \geq \eta_1$.
\eenmt
Moreover, we can also ensure that
\enmt
\item $\cW^{\chi}_{G, loc}(y) \cap \eta_1(x) = \eta_0(y)$ for $\mu$-a.e. $x$ and every $y \in \eta_1(x)$;
\item $f^{-1}(\eta_0(x))  = \eta_0(f^{-1}(x)) \cap f^{-1}(\eta_1(x))$ for $\mu$-a.e. $x$ and every $y \in \eta_1(x)$.
\eenmt

\begin{lemma}\label{lem: local entropy leq}
	We have
	\aryst
	h_{\mu}(f, \eta_1) \leq h_{\mu}(f, \eta_0) + \chi_F(k).
	\earyst
\end{lemma}
\begin{proof}
	This follows from \cite[Section 11]{LedYou2}.\end{proof}

\begin{rema}
	In Ledrappier-Young \cite{LedYou1}, this was proved in the setting where an invariant subfoliation of the unstable foliation is foliated by strong unstable foliation. Here neither $\cW^{\chi}_F$ or $\cW^\chi_G$ is not a strong subfoliation of $\cW^\chi$. But in our setting, the local product structure of $M^\a$ and the group action allows us to show that $\cW^\chi$ is $C^1$ foliated by both $\cW^\chi_F$ and $\cW^\chi_G$. This suffices for the construction of $\eta_0$, $\eta_1$.
\end{rema}

\begin{lemma} \label{lem: local entropy geq}
	We have
	\aryst
	h_{\mu}(f, \eta_1) = 2\chi_{G}(k).
	\earyst
\end{lemma}
\begin{proof}
	This is a consequence of \cite[Theorem 13.6]{BRHW1}, our hypothesis that $\mu^{\cW^\chi_F}$ is atomic, and the fact that $\pi_*\mu$ is the Haar measure on $G/\Gamma$.
\end{proof}

   \begin{cor}\label{lem: fiberislargerthangrouporbit}
   	If for $\mu$-a.e. $x$, $\mu^{G^\chi}_x$ is atmoic, then
	there exists a constant $\lambda > 1$ such that $\chi_F =  \lambda \chi_G$.
\end{cor}
\begin{proof}
	By the definition of $\chi$, there exists $\lambda > 0$ such that $\chi_F = \lambda \chi_G$. 
	Take an arbitrary $k \in \liealga$ such that $\chi(k) > 0$, and set $f= \ta(k)$.
	By the hypothesis in the lemma, we know that $h_{\mu}(f, \eta_0) = 0$.
	By Lemma \ref{lem: local entropy geq} and Lemma \ref{lem: local entropy leq}, we obtain
	\aryst
	\chi_F(k) \geq h_{\mu}(f, \eta_1) \geq 2 \chi_G(k).
	\earyst

\end{proof}

\subsection{Non-stationary normal form}

We recall a result in \cite{KK} on the existence of the non-stationary normal form. In our setting, their result states as follows.
\begin{lemma}\label{lem: nonstatnormform 1}
	For $\mu$-a.e. $x \in M^\a$, there exists a $C^{1+\epsilon}$ diffeomorphism $h_x : \cW^{\chi}_F(x) \to \R$ such that
	\enmt
	\item[$(i)$] $h_{\ta(k)x} \circ \ta(k) = D\ta(k) \circ h_{x}$ for every $k \in \liealga$,
	
	\item[$(ii)$] $h_x(x) = 0$ and $D_xh_x$ is an isometry,
	\item[$(iii)$] $h_x$ depends continuously on $x$ in the $C^{1+\epsilon}$ topology on a Pesin set.
	
	\eenmt
\end{lemma}

Let us denote by $\Omega$ the $\mu$-conull subset in Lemma \ref{lem: nonstatnormform 1} on which the non-stationary normal form is defined.
For any $x \in \Omega$, the map $h_x$ can be expressed in an explicit manner which we now describe.
We fix $x \in \Omega_0$ and an element $k_0 \in \liealga$ such that $\chi(k_0) < 0$, and denote $f = \ta(k_0)$. For any $z \in M^{\a}$, we denote
\aryst
Jf(z) = \norm{Df|_{E(z)}}.
\earyst
For any $y \in \cW^{\chi}_F(x)$, 
we have
\ary \label{eq: definition of nnf}
|h_x(y)| = \int_{x}^{y} \rho_x(z) dz
\eary
where 
\aryst
\rho_x(z) = \prod_{i=0}^{\infty} \frac{Jf(f^{i}(z))}{Jf(f^{i}(x))}.
\earyst
The integral in \eqref{eq: definition of nnf} is defined using the Riemannian metric on $\cW^{\chi}_F(x)$.
We define 
\aryst
\Omega_1 = \cup_{x \in \Omega} \cW^\chi_F(x).
\earyst
Then by \eqref{eq: definition of nnf}, we can define $h_y$ for any $y \in \Omega_1$. 
%\clr Since $\cW^\chi_F(x)$ is $1$-dimensional, and hence  orientable, we can define $h_y$ for every $y \in \Gamma_1$ such that for any $x \in \Gamma$,  for any $y_1,y_2 \in \cW^\chi_F(x)$, the map $h_{y_1} h_{y_2}^{-1}$ is orientable.\clb
We have the following useful observations.

\begin{lemma}\label{lem: chart 1}
	For any  $x \in \Omega$, for any $y_1,y_2 \in \cW^\chi_F(x)$, the map $h_{y_1} h_{y_2}^{-1}$ is an affine transformation of $\R$. 
\end{lemma}
\begin{proof}
	This is proved in \cite[Lemma 3.3]{KK}.
\end{proof}

\begin{lemma}\label{lem: chart 2}
	For any $y \in \Omega$, for any $z \in \Omega_1$ such that there exists $g \in G^\chi$ satisfying $z = \ta(g) y$, we have
	\aryst
	\ta(g)\cW^\chi_F(y) = \cW^\chi_F(z).
	\earyst
	Moreover, there exists $c \in \{\pm\norm{D\ta(g)|_{E(y)}}\} $ such that
	\aryst
	h_z \ta(g) h_y^{-1}(t) = c t, \quad \forall t \in \R.
	\earyst
\end{lemma}
\begin{proof}
	Take an arbitrary $w \in \cW^\chi_F(y)$, we denote $u = \ta(g)w$. Then for any $k \in \liealga$ such that $\chi(k) < 0$, we have
	\aryst
	\lim_{n \to \infty} n^{-1} \log d(\ta(nk)u , \ta(nk)z ) < 0.
	\earyst
	Moreover, we have
	\aryst
	\pi(u) = g \pi(w) = g \pi(y) = \pi(z).
	\earyst
	Thus we have $u \in \cW^\chi_F(z)$. This proves the first statement.
	
	We now prove the last statement.
	We use the natural parametrisation of $G^{\chi}$ by $\R^2$. Namely, we define a diffeomorphism
	$\theta_\chi: \R^2 \to G^{\chi}$ by
	\ary \label{item: thetachi}
	\theta_\chi(a,b) = {\rm Id} + (a+ib)E_{\chi}
	\eary
	where $E_{\chi}  = E_{s,t}$ if $\chi = \chi_{s,t}$.
	We write $g = \theta_\chi(v)$ for some $v \in \R^2$, and write $\lambda = e^{\chi_G(k_0)} < 1$. Notice that 
	\aryst
	f \ta(\theta_\chi(v)) = \ta(\theta_\chi(\lambda v)) f. 
	\earyst
	Thus we have
	\aryst
	Jf(u) &=&  \norm{D\ta(\theta_\chi(\lambda v))|_{E(f(w))}} Jf(w) \norm{D\ta(\theta_{\chi}(-v))|_{E(u)}}  \\
	&=&  Jf(w)   \norm{D\ta(\theta_\chi(\lambda v))|_{E(f(w))}}\norm{D\ta(\theta_{\chi}(v))|_{E(w)}}^{-1}. 
	\earyst
	More generally, for every integer $i \geq 0$, we have
	\aryst
	Jf(f^i(u)) =  Jf(f^i(w))   \norm{D\ta(\theta_\chi(\lambda^{i+1} v))|_{E(f^{i+1}(w))}}\norm{D\ta(\theta_{\chi}(\lambda^{i}v))|_{E(f^{i}(w))}}^{-1}. 
	\earyst
	Analogously, we have
	\aryst
	Jf(f^i(z)) =  Jf(f^i(y))   \norm{D\ta(\theta_\chi(\lambda^{i+1} v))|_{E(f^{i+1}(y))}}\norm{D\ta(\theta_{\chi}(\lambda^{i}v))|_{E(f^{i}(y))}}^{-1}. 
	\earyst
	To simplify notation, we set
	\aryst
	\xi_{i,w} &=& \norm{D\ta(\theta_{\chi}(\lambda^{i}v))|_{E(f^{i}(w))}}, \\
	\xi_{i,y} &=& \norm{D\ta(\theta_{\chi}(\lambda^{i}v))|_{E(f^{i}(y))}}.
	\earyst
	Notice that $\xi_{i,w}, \xi_{i,y}$ tend to $1$ exponentially fast as $i$ tends to infinity.
	Then for any $w_* \in \cW^\chi_F(y)$, denote $u_* = \ta(g) w_*$, we have
	\aryst
	|h_z(u_*)| &=& \int_{z}^{u_*} \rho_z(u) du \\
	&=& \int_{z}^{u_*} \prod_{i=0}^{\infty} \frac{Jf(f^{i}(u))}{Jf(f^{i}(z))} du \\
	&=& \int_{z}^{u_*} \prod_{i=0}^{\infty} [ \frac{Jf(f^{i}(w))}{Jf(f^{i}(y))} \frac{\xi_{i+1,w} \xi_{i,y}}{\xi_{i+1,y} \xi_{i,w}} ] du \\
	&=& \int_{z}^{u_*} \rho_y(w) \frac{\xi_{0,y}}{\xi_{0,w}} du \\
	\mbox{( as $u =\ta(g)w$ )}&=& \xi_{0,y} \int_{y}^{w_*} \rho_y(w) dw \\
	&=& \norm{D\ta(g)|_{E(y)}} |h_{y}(w_*)|.
	\earyst
	This confirms the last statement.
\end{proof}

\subsection{The proof for the atomic case}
We use the following parametrisation of $\cW^{\chi}$. For every $x \in \Omega_1$, we define the map $H_x$ from $\cW^{\chi}(x)$ to $\R^3$ by 
\aryst
H_x(p)= (a(p), b(p)) \mbox{ if we have }p = \ta(\theta_{\chi}(a(p)))h_{x}^{-1}(b(p))
\earyst
where $\theta_\chi$ is defined in \eqref{item: thetachi}.
It is straightforward to verify that $H_x$ is a homeomorphism.

We notice that for any $x \in \Omega_1$, for any $k \in \liealga$, there exists $c \in \{ \pm \norm{D\ta(k)|_{E(x)}}\}$ such that 
\ary \label{item: graphtransfer} 
H_{\ta(k)x} \ta(k) H_x^{-1}(u,v) = (e^{\chi_G(k)}u, cv), \quad \forall u \in \R^2, v \in \R.
\eary

Let us define a subgroup of the affine transformations of $\R^3$ as follows,
\aryst
\bbA = \{(v_1,v_2,v_3) \mapsto (v_1+a_1,v_2+a_2, b v_3+c) \mid a_1,a_2, c \in \R, b \in \R_* \}.
\earyst
For each $T \in \bbA$, we will use $a_1(T), a_2(T), b(T), c(T)$ to denote the coefficients in the expression of $T$. We also set 
\aryst
a(T) := ( a_1(T), a_2(T)).
\earyst

We collect some useful properties of $H_x$.
\begin{lemma}\label{chart prop 1}
	For $\mu$-a.e. $x$, for $\mu^{\cW^\chi}_x$-a.e. $y$, the map $H_y H_x^{-1}$ belongs to $\bbA$. Moreover, we have
	\aryst
	\pi(y) = \ta(\theta_\chi(-a(H_yH_x^{-1})) ) \pi(x).
	\earyst
\end{lemma}
\begin{proof}
	As $\Omega$ is $\mu$-conull, for $\mu$-a.e. $x$, $\mu^{\cW^\chi}_x$-a.e. $y$ belongs to $\Omega$. We fix $x, y \in \Omega$ as above. Since $\pi(y) \in G^\chi(\pi(x))$, we see that there exists $v \in \R^2$ such that $z := \ta(\theta_\chi(v))(x) \in \pi^{-1}(\pi(y)) \cap \cW^\chi(x)$. Then it is clear that 
	$z  \in \cW^\chi_F(y)$.
	By Lemma \ref{lem: chart 1}, we can see that $H_{y}H_{z}^{-1} \in \bbA$. Moreover, it is clear that $a(H_y H_z^{-1}) = (0,0)$.
	
	By Lemma \ref{lem: chart 2}, we have
	\ary \label{item: h_zh_x}
	h_{z}^{-1}(c t) = \ta(\theta_\chi(v)) h_x^{-1}(t), \quad \forall t \in \R
	\eary
	where $c \in \{\pm \norm{D\ta(\theta_\chi(v))|_{E(x)}}\}$.
	As $H_x, H_z$ are homeomorphisms between $\cW^{\chi}(x)$ and $\R^3$, for any $s \in \R^2$, $t \in \R$, there exists a unique pair  $s' \in \R^2$, $t' \in \R$ such that $H_x^{-1}(s,t) = H_z^{-1}(s',t')$. Then by the definitions of $H_x, H_z$ and by \eqref{item: h_zh_x}, we have
	\aryst
	\ta(\theta_\chi(s))h_x^{-1}(t) &=& \ta(\theta_\chi(s')) h_z^{-1}(t') \\
	&=&   \ta(\theta_\chi(s'))  \ta(\theta_\chi(v)) h_x^{-1}(c^{-1}t') \\
	&=&   \ta(\theta_\chi(s' + v))  h_x^{-1}(c^{-1}t').
	\earyst
	Consequently, we have
	\aryst
	s' = s - v, \quad t' = c t.
	\earyst
	Thus $H_z H_x^{-1} \in \bbA$ and $a(H_zH_x^{-1}) = -v$. Hence $H_y H_x^{-1} \in \bbA$ and $a(H_yH_x^{-1})=-v$. This concludes the proof.
\end{proof}

We denote
\ary \label{lab: identify bbA0}
\bbA^{0} &=& {\rm Ker}(p)  \\
&=&  \{(v_1,v_2,v_3) \mapsto (v_1,v_2,b v_3+c \mid b,c\in \R\}.\nonumber
\eary

We denote by  $\cal{PM}(\R^3)$ the space of equivalence classes under proportionality of Radon measure on $\R^3$. 
We define 
\aryst
\cal H := L^0(\R^2, Leb).
\earyst
That is, the set of Borel measurable $\R$-valued functions on $\R^2$ modulo the equivalence 
\aryst
f_1 \sim f_2 \mbox{ iff $f_1(v)= f_2(v)$ for Lebesgue almost every $v$.} 
\earyst
It is well-known that $\cal H$, equipped with the topology given by convergence in measure, is a complete metric space.

\begin{proof}[Proof of Proposition \ref{main tech prop} --- the atomic case]
	We assume by contradiction that $\mu^{G^{\chi}}_x$ is atomic for $\mu$-a.e. $x$. 
	Consequently,
	$\mu^{\cW^\chi_G}_x$ is atomic for $\mu$-a.e. $x$.
	
	We let $\{ [\mu^{\cW^\chi}_x] \}_{x \in M^\a}$ be defined in subsection \ref{subsec: Conditional measure} where $\mu^{\cW^\chi}_x$ is a Radon measure on $\cW^\chi(x)$ determined up to a scalar.
	For $\mu$-a.e. $x$, we define
	\ary \label{item: psidef}
	\Psi(x) = [(H_x)_*(\mu^{\cW^{\chi}}_x)] \in \cal{PM}(\R^3).
	\eary
	We have the following.
	\begin{lemma} \label{lem. dila}
		There exists a unique $r \in \cal H$ such that the following is true.
		Fix an arbitrary constant $u > 0$ and an arbitrary Radon measure $\omega \in \Psi(x)$. For every $c > 0$, we define
		\aryst
		\omega_c := \omega|_{B_{\R^2}(0,u) \times (-c,c)} \in \cal M(\R^3).
		\earyst
		Let $\pi_{1,2} : \R^3 \to \R^2$ denote the projection onto the first two coordinates of $\R^3$. 
		Then the measure  
		\aryst
		\bar \omega_c : = (\pi_{1,2})_*\omega_c \in \cal M(\R^2)
		\earyst
		is absolutely continuous with respect to the Lebesgue measure on $\R^2$;
		and we have
		\aryst
		\omega_c = \int_{r^{-1}(-c,c)}  \delta_{(v,r(v))} d\bar \omega_c(v).
		\earyst
	\end{lemma}
	\begin{proof}
		We assume for simplicity that $u = 1$, and we will define $r$ over $B_{\R^2}(0,1)$. The general case is similar.
		
		Given $d> 0$. We deduce that $\bar\omega_d$ is absolutely continuous with respect to the Lebesgue measure by the fact that $\pi_*\mu$ is the Haar measure on $G / \Gamma$. To simply notation, we set
		\aryst
		R_d  = \{  v \mid \frac{d\bar\omega_d}{dLeb}(v) > 0 \}.
		\earyst
		We have $R_d \subset R_{d'}$ for any $d < d'$, and $\cup_{d > 0} R_d$ coincides with $B_{\R^2}(0,1)$ up to a Lebesgue null set. 
		
		By Rokhlin's disintegration theorem, we obtain
		\aryst
		\omega_d = \int_{\R^2}  \omega_d^{\{v\} \times \R }d\bar\omega_d(v),
		\earyst 
		where $\omega_d^{\{v\} \times \R }$ denotes the conditional measure of $\omega_d$ on $\{v\} \times \R$.
		As we know that $\mu^{\cW^\chi_F}_y$ is the Dirac measure at $y$ for $\mu$-a.e. $y$; and that for $\omega$-a.e. $(v,s) \in \R^3$,
		\aryst
		\omega_d^{\{v\} \times \R} \leq (H_x)_*( \mu^{\cW^\chi_F}_{H_x^{-1}(v, s)}),
		\earyst
		we can conclude that $\omega_d^{\{v\} \times \R}$ is a Dirac measure on $\{v\} \times \R$ for $\bar\omega_d$-a.e. $v \in R_d$.
		Thus there exists an essentially unique $\bar\omega_d$-a.e. defined measurable function $r_d:  R_d \to (-d,d)$ such that
		\aryst
		\omega_d = \int_{R_d}  \delta_{(v, r_d(v))} d\bar\omega_d(v).
		\earyst
		We extend $r_d$ to a measurable function from $B_{\R^2}(0,1) $ to $(-d,d)$ by setting 
		\aryst
		r_d |_{B_{\R^2}(0,1)  \setminus R_d} \equiv 0.
		\earyst
		
		By definition, for every $c \in (0,d)$, we have
		\aryst
		\omega_c = \int_{r_d^{-1}(-c,c)} \delta_{(v,r_d(v))} d\bar\omega_d(v).
		\earyst
		Consequently, we have
		\aryst
		\bar \omega_c = \bar \omega_d |_{r_d^{-1}(-c,c)}
		\earyst
		and
		\aryst
		r_c = r_d |_{r_d^{-1}(-c,c)}.
		\earyst
		We let $r : B_{\R^2}(0,1) \to \R$ be the pointwise limit of $r_d$ as $d$ tends to infinity.
		It is straightforward to verify that $r$ satisfies the requirement of the lemma.
	\end{proof}
	
	For a $\mu$-typical $x$, we let $r$ be given by Lemma \ref{lem. dila}, and set 
	\aryst
	S(x) = r.
	\earyst
	
	\begin{cor} \label{cor. graphtransformofs}
		For $\mu$-a.e. $x$, for $\mu^{\cW^\chi}_x$-a.e. $y$, we have
		\aryst
		Graph(S(y)) = (H_y H_x^{-1}) Graph(S(x)). 
		\earyst
	\end{cor}
	\begin{proof}
		For $x,y$ in the corollary, there exists a constant $c > 0$ such that  
		\aryst
		\mu^{\cW^\chi}_x = c \mu^{\cW^\chi}_y.
		\earyst
		Then by \eqref{item: psidef}, we have
		\aryst
		\Psi(y)=(H_y H_x^{-1})_*\Psi(x).
		\earyst
		By Lemma \ref{lem. dila}, we see that $\Psi(x)$, resp. $\Psi(y)$, is supported on the graph of $S(x)$, resp. $S(y)$.
		The corollary then follows suit.
	\end{proof}	
	
	We define for every $c > 0$ that
	\ary \label{lab. pc}
	P_c(x) := S(x)^{-1}(-c,c) \cap B_{\R^2}(0,1).
	\eary
	\begin{rema}
	    By definition, it is clear that 
		\ary
		\lim_{c \to +\infty} Leb(P_c(x)) = Leb(B_{\R^2}(0,1)) = \pi.
		\eary
	\end{rema}
	
	\begin{lemma} \label{lem. dilation}
		We have
		\aryst
		\lim_{c \to 0} Leb(P_c(x)) = 0.
		\earyst
	\end{lemma}
	\begin{proof}
		It is clear that we have
		\aryst
		\lim_{c \to 0} Leb(P_c(x)) = Leb(S(x)^{-1}(0) \cap B_{\R^2}(0,1)).
		\earyst
		If $Leb(S(x)^{-1}(0) \cap B_{\R^2}(0,1)) > 0$, then we see that for any $d > 0$ and any $\omega \in \Psi(x)$, the conditional measure of $\omega_d$ on $\R^2 \times \{0\}$ is not atomic. On the other hand, by definition, we see that for  a $\mu$-typical $x$, we have
		\aryst
		\omega_d^{\R^2 \times \{0\}} \leq [ (H_x)_*(\mu^{\cW^\chi}_{x}) ]^{\R^2 \times \{0\}}= (H_x)_*\mu^{\cW^{\chi}_G}_x.
		\earyst 
		By our hypothesis, $\mu^{\cW^{\chi}_G}_x$ is atomic.
		This is a contradiction.
	\end{proof}

	We define $\lambda : M^{\a} \to \R$ as follows,
	\aryst
	\lambda(x) = \inf \{ c > 0 \mid
	Leb(P_c(x)) \geq \frac{1}{2} \}.
	\earyst
	
	By Lemma \ref{lem. dilation}, we see that $\lambda(x) \in (0, \infty)$.
	
	For any real constant $c \neq 0$, we define map $D_c: \R^3 \to \R^3$ as 
	\aryst
	D_c(a,b) = (a, c^{-1}b), \quad \forall a \in \R^2, \forall b \in \R.
	\earyst
	We define
	\aryst
	\Phi(x) &=& (D_{\lambda(x)})_*\Psi(x), \\
	\hat S(x) &=& \lambda(x)^{-1}S(x).
	\earyst
	
	By definition, for any $t \in L_\chi$, we have $\chi_G(t) = 0$. By \eqref{item: invarianceconditionalmeasure}, \eqref{item: psidef}, \eqref{item: graphtransfer},
	for any $t \in L_{\chi}$, we have 
	\ary \label{item: conditionalmeasure transfer2}
	\Psi(\ta(t)x) =  (D_d)_*\Psi(x)
	\eary for certain constant $d \neq 0$. Then by definition, we have
	\aryst
	S(\ta(t)x) = d^{-1}S(x).
	\earyst 
	
	Take an arbitrary $\lambda' > \lambda(x)$. 
	Notice that by the definition of $P_c$ and \eqref{item: conditionalmeasure transfer2}, we have
	\aryst
	P_c(\ta(t)x) = P_{cd}(x), \quad \forall c > 0.
	\earyst
	Then by \eqref{item: conditionalmeasure transfer2} we have
	\aryst
	P_{d^{-1} \lambda ' } (\ta(t)x)  = P_{\lambda '}(x)  \geq \frac{1}{2}.
	\earyst
	Consequently, we have 
	\aryst
	d^{-1} \lambda(x) \geq \lambda(\ta(t)x).
	\earyst
	By symmetry, we can also show that  $d^{-1} \lambda(x) \leq \lambda(\ta(t)x)$. Thus $\lambda(\ta(t)x) = d^{-1} \lambda(x)$.
	By definition, we see that
	\ary \label{eq. phiishchiinv}
	\Phi(x) = \Phi(\ta(t)x) \mbox{ and }\hat S(x) = \hat S(\ta(t)x) .
	\eary
	As $t$ is  an arbitrary element of $L_\chi$, we see that $\hat S$ is an $H_{\chi}$-invariant function (modulo $\mu$).
	We set 
	\aryst
	\cA = \hat S^{-1} \cB_{\cal H}.
	\earyst
	
	For any closed subgroup $H \subset G$, we denote by 
	$\cE_{H}$ the $\sigma$-algebra generated by  $H$-invariant sets modulo $\mu$. More precisely, we define
	\aryst
	\cE_{H} := \{B \in \cB_{M^{\a}} \mid g^{-1}B = B \mbox{ mod } \mu, \quad \forall g \in H \}.
	\earyst
		It is well-known that (for example, see \cite[Theorem 6.1]{bookEW}), if $\mu$ is $H$-invariant, then for $\mu$-a.e. $x$, the atom $\cE_{H}(x)$ is the $H$-ergodic component of $\mu$ at $x$.
	By \eqref{eq. phiishchiinv}, we have that 
	\aryst
	\cA \subset \cE_{H_{\chi}}.
	\earyst

    By $\chi \in   \Sigma^{out}_3 \cap (-\Sigma^{non}_Q)$, we see that
    \aryst
    [\cW^{-\chi}]  = [\cW^{-\chi}_G] \supset \cE_{H_{\chi}}.
    \earyst
	By the similar argument as in Section \ref{sec. nonatomic}, we deduce that
	\aryst
	[\cW^\chi] \supset \cE_{H_{\chi}}.
	\earyst

	 Consequently, for $\mu$-a.e. $x$, for $\mu^{\cW^\chi}_x$-a.e. $y$, we have $y \in \cA(x)$, or in another words,
	 \aryst
	  \hat S(x) = \hat S(y).
	 \earyst 
	 The consideration of $\hat S$ is related to the method presented in \cite{EL-asymmetry}.
	 
	The above discussion shows that for a  $\mu$-typical point $x$, the set
	\aryst
	U := \{y \in \cW^{\chi}(x) \mid \hat S(x) = \hat S(y) \}
	\earyst
	satisfies that $\pi_{1,2}U$ is  non-discrete. 
	%We notice that for each $y \in U$, we have
	%\aryst
	%(D_{\lambda(x)}H_x)^* S(x) \sim 
	%(D_{\lambda(y)}H_y)^* S(y) = (D_{\lambda(y)} H_y H_x^{-1})^* (H_x)^* S(y).
	%\earyst
	%As $y \in \cW^{\chi}(x)$, we have $\mu^{\cW^{\chi}}_y = \mu^{\cW^{\chi}}_x$.
	By Corollary \ref{cor. graphtransformofs}, for any $y \in U$ we have
	\aryst
	(D_{\lambda(x)^{-1}\lambda(y)} H_y H_x^{-1}) Graph(S(x)) = Graph(S(x)).
	\earyst

	We set
	\aryst
	\bbA_x = \{T \in \bbA \mid T Graph(S(x)) = Graph(S(x)) \}.
	\earyst
	We notice that $\bbA_x$ has a natural factor, denoted by $p : \bbA_x \to \overline{\bbA}_x$, where
	\aryst
	\overline{\bbA}_x = \{ \check{T} : \R^2 \to \R^2  \mid \exists T \in \bbA_x \mbox{ such that }  \check{T} \pi_{1,2} = \pi_{1,2} T \}
	\earyst
	and as before $\pi_{1,2}$ denotes the projection from $\R^3$ onto its first two coordinates. 
	We can naturally identify $\overline{\bbA}_x$ with a subset of $\R^2$ by taking the translation vector.
	
	We set  $\bbA^{0}_x = \bbA^0 \cap \bbA_x$. By definition, $\bbA_x$, $\bbA^0$, $\bbA^0_x$ are closed subgroups of $\bbA$, and there  is an exact sequence
	\aryst
	0 \xrightarrow{}   \bbA^0_x \xrightarrow{} \bbA_x \xrightarrow{} \overline{\bbA}_x \xrightarrow{} 0.
	\earyst

	We notice  the following.
	\begin{lemma}\label{lem. kerneltrivial}
		We have 
		\enmt
		\item $\bbA^0_x = \{{\rm Id}\}$;
		\item the map $\bbA_x \to \overline{\bbA}_x$ is proper. Consequently, $\overline{\bbA}_x$ is closed.
		\eenmt
	\end{lemma}
	\begin{proof}
		We denote $r = S(x)$.
		Take an arbitrary $T \in \bbA_x$. By the uniqueness of $r$ in Lemma \ref{lem. dila}, we can see that
		\aryst
		b(T) r(v-a(T)) + c(T) = r(v)
		\earyst
		for Lebesgue almost every $v \in \R^2$.
		
		If $\bbA^0_x \neq \{{\rm Id}\}$ and ${\rm Id} \neq T \in \bbA^0_x$, then $r$ must equal to a constant Lebesgue almost everywhere. This contradicts the our hypothesis that $\mu^{G^\chi}$ is atomic almost everywhere.  Item (1) follows suit.
		
		As we have seen $r$ is not almost everywhere constant, there exist disjoint intervals $I_1, I_2 \subset \R$ such that $r^{-1}(I_i)$ has positive Lebesgue measure for $i=1,2$.
		
		Let $\{ T_n \}_{n \geq 0}$ be a sequence in $\bbA_x$ such that 
		\aryst
		\lim_{n \to \infty}a(T_n) = 0.
		\earyst
		Then for all sufficiently large $n$,  for $i=1,2$, we may find $v_{n,i} \in r^{-1}(I_i)$ such that $v_{n,i} -a(T_n) \in r^{-1}(I_i)$. 
		Thus
		\aryst
		b(T_n)(r(v_{n,1} - a(T_n)) - r(v_{n,2}-a(T_n))) = r(v_{n,1}) - r(v_{n,2}).
		\earyst
		This implies that for all sufficiently large $n$ we have
		\aryst
		|b(T_n)| \leq dist(I_1, I_2)^{-1} diam(I_1 \cup I_2).
		\earyst
		In a similar way, we may bound $c(T_n)$ for all sufficiently large $n$.
		This implies the properness of the map from $\bbA_x$ to $\overline{\bbA}_x$.
	\end{proof}

	By Lemma \ref{lem. kerneltrivial}(1),  we may define $b(z) := b(T)$ and $c(z) := c(T)$ for every $z \in \overline{\bbA}_x$ where $T$ is the unique  element of $\bbA_x$ with $a(T) = z$.

	By Lemma \ref{lem. kerneltrivial}(2), we conclude that $\overline{\bbA}_x$ is a closed, non-discrete subgroup of the translations on $\R^2$. Thus $\overline{\bbA}_x$ is a linear subspace of $\R^2$ of positive dimension.

	It is direct to verify that $b(z_1+z_2) = b(z_2) b(z_1)$ for any $z_1,z_2 \in \overline{\bbA}_x$. 
	Then there exists a linear functional $\ell^x: \overline{\bbA}_x \to \R$ such that $b(z) = e^{\ell^x(z)}$ for any $z \in \overline{\bbA}_x$.

    Assume that for $\mu$-a.e. $x$, we have $\ell^x \neq 0$. We take a $\mu$-typical $x$, and abbreviate $\ell^x$ as $\ell$.
	Take two arbitrary elements $T_1,T_2 \in \bbA_x$, and some $v \in \overline{\bbA}_x$, $u \in \R$. To simply notation, we set $z_i = a(T_i)$ for $i=1,2$. Then we have
	\aryst
	T_2 T_1 (v,u) &=& T_2(v+z_1, e^{\ell(z_1)}u + c(z_1)) \\
	&=& (v + z_1 + z_2, e^{\ell(z_2)}(e^{\ell(z_1)}u  + c(z_1) )+  c(z_2) ) \\
	&=& (v + z_1 + z_2, e^{\ell(z_2+z_1)}u + ( e^{\ell(z_2)} c(z_1) + c(z_2) ) ).
	\earyst
	We can see that for any $z_1,z_2 \in \overline{\bbA}_x$,
	\ary
	c(z_1+z_2) &=& e^{\ell(z_2)} c(z_1) + c(z_2). \label{lab: phiequation}
	\eary
	
	Then by \eqref{lab: phiequation}, we obtain
	\ary \label{lab: phicaseII}
	c(z) = c_0(e^{\ell(z)}- 1), \quad \forall z \in \overline{\bbA}_x
	\eary
	for some constant $c_0 \in \R$.

	By  \eqref{lab: phicaseII}, we see that for $\mu$-a.e. $x$, there exists a linear functional $ \ell^x : \overline{\bbA}_x \to \R$,
	and a constant $c^x_0 \in \R$ such that
	\aryst
	c_x(z) =  c^x_0(e^{\ell^x(z)}- 1), \quad \forall z \in \overline{\bbA}_x.
	\earyst
	
	For a $\mu$-typical $x$, for any $k \in \liealga$, and for any $z \in \overline{\bbA}_x$, we set
	\aryst
	C^x_{k,\pm}(v,u) = (e^{\chi_G(k)}v,  \pm \norm{D\ta(k)|_{E(x)}}u)
	\earyst
	and
	\aryst
	T^x_z(v,u) = (v+z, e^{\ell^x(z)}u + c^x_0(e^{\ell^x(z)}-1)   ) .
	\earyst
	
	By \eqref{item: graphtransfer} and straightforward computations, we deduce that  for any $\sigma \in \{-, +\}$,
	\aryst
		C^x_{k, \sigma} T^x_z (C^x_{k, \sigma})^{-1}(v,u) &=& C^x_{k, \sigma} T^x_z(e^{-\chi_G(k)}v,  \sigma \norm{D\ta(k)|_{E(x)}}^{-1}u) \\
		&=& C^x_{k, \sigma} (e^{-\chi_G(k)}v + z,  e^{\ell^x(z)}\sigma \norm{D\ta(k)|_{E(x)}}^{-1} u + c^x_0(e^{\ell^x(z)}-1 ) ) \\
	    &=& (v +e^{\chi_G(k)}z ,  e^{\ell^x(z)}  u + \sigma \norm{D\ta(k)|_{E(x)}}c^x_0(e^{\ell^x(z)}-1 ) ).
	\earyst
	It is direct to see that
	\aryst
		C^x_{k, \sigma} T^x_z (C^x_{k, \sigma})^{-1}
	\in \bbA_{\ta(k)x}.
	\earyst
	Then for some $\sigma \in \{-,+\}$ we have for any $z \in \overline{\bbA}_x$ that
	\ary \label{eq. ctct}
	C^x_{k, \sigma} T^x_z C^x_{k, \sigma} = T^{\ta(k)x}_{e^{\chi_G(k)}z}.
	\eary
	Then we have
	\aryst
	\ell^x(z) = \ell^{\ta(k)x}(e^{\chi_G(k)}z).
	\earyst
	By this is impossible by Poincar\'e's recurrence lemma and our hypothesis that $\ell^x \neq 0$ for $\mu$-a.e. $x$.
	Consequently, for $\mu$-a.e. $x$, we have $\ell^x \equiv 0$.
	Then it is easy to see that $c$ is a linear functional on $\overline{\bbA}_x$, which we denote by $c_x$.
	Again by \eqref{eq. ctct}, we deduce that for some $\sigma \in \{\pm 1\}$,
	\aryst
	c_{\ta(k)x} = \sigma \norm{D\ta(k)|_{E(x)}} e^{-\chi_G(k)} c_x.
	\earyst
	Consequently, we have
	\ary \label{eq. c1c2}
		\norm{c_{\ta(k)x}} = \norm{D\ta(k)|_{E(x)}} e^{-\chi_G(k)} \norm{c_x}.
    \eary
    
    Assume that $c_x \neq 0$ for $\mu$-a.e. $x$.  
    Notice that
    \aryst
    \lim_{n \to \infty} n^{-1} \log \norm{D\ta(nk)|_{E(x)}} = \chi_F(k), \quad \forall k \in \liealga.
    \earyst
    By Corollary \ref{lem: fiberislargerthangrouporbit}, we have $\chi_F = \lambda \chi_G$ for some $\lambda > 1$.
    We get a contradiction by \eqref{eq. c1c2} and Poincar\'e's recurrence theorem.
    
    Thus we have proved that $c_x \equiv 0$ for $\mu$-a.e. $x$, and as a result,
    \aryst
    \bbA_x = \{ ((v,u) \mapsto (v + z, u)) \mid z \in \overline{\bbA}_x \}.
    \earyst

	However, for any Radon measure $\omega$ on $\R^3$ satisfying that
	\aryst
	T_* \omega = \omega, \forall T \in \bbA_x,
	\earyst
	we know that for $\omega$-a.e. $(v,u) \in \R^3$ where $v \in \R^2$ and $u \in \R$,
	the conditional measure  of $\omega$ on $\R^2 \times \{u\}$ is nonatomic. While this contradicts our hypothesis that  $\mu^{\cW^\chi_G}$ is atomic.
\end{proof}

\noindent{\bf{Acknowledgements.}}
I am indebted to Jinpeng An for useful inputs from Lie theory.
I thank Federico Rodriguez-Hertz for remark on a technical point in \cite{KKRH}. I thank Jinxin Xue and Lei Yang for many useful and stimulating discussions on dynamics. This work is initiated during my stay in KTH Royal Institute of Technology. I thank their hospitality.

\end{document}